\documentclass[12pt]{amsart}
\usepackage{amssymb}
\usepackage{latexsym}
\usepackage{amsmath}
\usepackage{enumitem}
\usepackage{euscript}

   \def\sH{{\mathfrak H}}   
      \def\sL{{\mathfrak L}}
\def\sM{{\mathfrak M}}

      \def\dC{{\mathbb C}}
\def\dD{{\mathbb D}}

   \def\dN{{\mathbb N}}   
      \def\dR{{\mathbb R}}

\def\cA{{\mathcal A}}   \def\cB{{\mathcal B}}

      \def\cL{{\mathcal L}}
\def\cM{{\mathcal M}}   \def\cN{{\mathcal N}}   
   \def\cQ{{\mathcal Q}}   \def\cR{{\mathcal R}}
\def\cS{{\mathcal S}}   \def\cT{{\mathcal T}}   \def\cU{{\mathcal U}}
      \def\cX{{\mathcal X}}
   \def\cZ{{\mathcal Z}}

\def\ran{{\rm ran\,}}
\def\cran{{\rm \overline{ran}\,}}
\def\dom{{\rm dom\,}}

\newtheorem{theorem}{Theorem}[section]
\newtheorem{lemma}[theorem]{Lemma}
\newtheorem{proposition}[theorem]{Proposition}
\newtheorem{corollary}[theorem]{Corollary}
\newtheorem{definition}[theorem]{Definition}
\newtheorem{remark}[theorem]{Remark}

\numberwithin{equation}{section}

\def\RE{{\rm Re\,}}
\def\IM{{\rm Im\,}}

\def\wt{\widetilde}
\def\wh{\widehat}

\def\f{\varphi}
\def\half{{\frac{1}{2}}}
\def\uphar{{\upharpoonright\,}}

\begin{document}
\title
[The Kato square root problem]
{On the Kato square root problem}
\author[Yury Arlinski\u{\i}]{Yury Arlinski\u{\i}}
\address{Stuttgart, Germany}
\email{yury.arlinskii@gmail.com}


\vskip 1truecm
\thispagestyle{empty}

\date{\today}
\begin{abstract} In the infinite-dimensional separable complex Hilbert space we construct new abstract examples of unbounded maximal accretive and maximal
sectorial operators $B$ for which $\dom B^{\frac{1}{2}}\ne\dom B^{*{\frac{1}{2}}}$. New criterions for the equality are established.
\end{abstract}
\maketitle
\section{Introduction}

Let $\sH$ be an infinite-dimensional separable complex Hilbert space and let $B$ be a maximal accretive operator \cite{Kato}. As is well known (see, e.g., \cite{haase, Kato1961, Langer1962, MP1962, Ouhabaz, SF}) for each $\gamma\in(0,1)$ the maximal accretive fractional power $B^\gamma$ can be defined.
For instance, one can use
\begin{itemize}
\item the Balakrishnan representation \cite{balakr}
\begin{equation}\label{lhjcnt}
B^\gamma u=\cfrac{\sin(\gamma\pi)}{\pi}\int\limits_{0}^\infty t^{\gamma-1}B(B+tI)^{-1}u dt,\; u\in\dom B;
\end{equation}
\item the Sz.-Nagy -- Foias functional calculus for contractions \cite[Chapter IV]{SF}:

if $T:=(I-B)(I+B)^{-1}$, then
\[
B^\gamma=v^\gamma(T)(u^{\gamma}(T))^{-1},\; v(z)=1-z,\; u(z)=1+z,\;z\in\dD,
\]
i.e., $B^\gamma=\left((I-T)(I+T)^{-1}\right)^\gamma=(I-T)^\gamma\left((I+T)^{\gamma}\right)^{-1}.$
\end{itemize}

The fractional powers possess the properties \cite{Kato1961, SF}:
\begin{itemize}
\item the operator $B^\gamma$ is \textit{regularly accretive} with index $\le \tan{\frac{\pi\gamma}{2}}$ \cite{Kato1961}, i.e., $B^\gamma$ is \textit{maximal sectorial} with vertex at the origin and with semi-angle $\cfrac{\pi\gamma}{2}$ \cite{Kato};
\item $\dom (B+\lambda I)^\gamma=\dom B^\gamma$ for all $\lambda>0;$

\item $(B^*)^\gamma=(B^\gamma)^*$ (in the sequel $(B^*)^\gamma$ we will  denote by  $B^{*\gamma}$);
\item if $\gamma\in(0, \half)$, then $\dom B^\gamma=\dom B^{*\gamma}$ and the operator $\RE(B^\gamma):=\half(B^\gamma+B^{*\gamma})$ is selfadjoint;

\item for each $\gamma\in(\half,1)$ there is an example of $B$ such that $\dom B^\gamma\ne\dom B^{*\gamma} ;$
\item for the square root $B^{\half}$ ($\gamma=\half$) it is proved in \cite[Theorem 5.1]{Kato1961} that the intersection $\dom B^{\half}\cap\dom B^{*\half}$
is a core of both $B^{\half}$ and $B^{*\half} $ and the operator
$$\RE (B^{\half})=\half(B^\half+B^{*\half}),\;\dom \RE(B^\half)= \dom B^{\half}\cap\dom B^{*\half} $$
 is selfadjoint.
\end{itemize}
Other remarkable properties of powers $B^\gamma$ for $\gamma\in(0,\half)$ can be found in \cite[Theorem 1.1, Theorem 3.1]{Kato1961} and in \cite[Chapter IV, Theorem 5.1]{SF}.

The following problem was formulated by Kato in \cite{Kato1961}: \textit{is the equality
\begin{equation}\label{jcyjdyt}
\dom B^\half=\dom B^{*\half}
\end{equation}
 true for an arbitrary unbounded  maximal accretive operator $B$?}

The negative answer was given in \cite{Lions1962} by Lions. By means of the theory of interpolation spaces he showed that for the maximal accretive operator in $L^2(\dR_+)$
\begin{equation}
\label{lionsop}
\cB_0=\cfrac{d}{dx},\;\dom \cB_0=\left\{f\in H^{1}(\dR_+): f(0)=0\right\}
\end{equation}
holds the inequality $\dom \cB_0^\half\ne\dom \cB_0^{*\half}$.
Note that the operator $\cB_0$ is not sectorial and, moreover, the operator $-i\cB_0$ is maximal symmetric and non-selfadjoint. 

The next results were established in \cite{Lions1962} and \cite {Kato1962}.
\begin{theorem}\label{lionskato}
Let $B$ be a maximal sectorial operator and let ${\mathfrak b}$ be the closed quadratic form associated with $B.$
 Then
 
 1) the inclusion $\dom B^{\half}\subseteq\dom {\mathfrak b}$ is equivalent to the inclusion $\dom B^{*\half}\supseteq\dom {\mathfrak b}$; the same is true when $B$ and $B^*$ are exchanged;

2) the inclusion $\dom B^{\half}\subseteq\dom B^{*\half}$ is equivalent to the inclusions $\dom B^{\half}\subseteq\dom {\mathfrak b}\subseteq\dom B^{*\half}$;
the same is true when $B$ and $B^*$ are exchanged.
\end{theorem}
It follows \cite{Lions1962}, \cite {Kato1962} that if both $\dom B^{\half}$ and $\dom B^{*\half}$ are subsets (or oversets) of $\dom {\mathfrak b}$, then
\begin{equation}\label{dvarav}
\dom B^\half=\dom B^{*\half}=\dom {\mathfrak b},
\end{equation}
 and therefore
 \[
{\mathfrak b}[u,v]=(B^{\half}u, B^{*\half}v),\; u,v\in \dom B^\half= \dom B^{*\half}=\dom {\mathfrak b}.
\]
Besides, it is shown in \cite{Lions1962}, \cite {Kato1962} that the equality $\dom B=\dom B^*$ implies \eqref{dvarav}.

McIntosh in \cite{McInt1972} (see also \cite[Preliminaries, Theorem 6]{AuTch}) presented an abstract example of unbounded maximal sectorial operator $B$ such that $\dom B^\half\ne\dom B^{*\half} .$ The operator constructed in \cite{McInt1972} is the countable orthogonal sum of special finite--dimensional sectorial operators with a fixed semi-angle.

 By means of another way and by using the operator $\cB_0$ in \eqref{lionsop}, Gomilko in \cite{Gomilko} has proved  the existence of maximal sectorial operators $B$ with the property
$\dom B^\half\ne\dom B^{*\half} .$  In \cite{Gomilko} it is shown that if a maximal accretive operator $B$ is boundedly invertible, then the operator $B^\half$ admits the representation
\begin{equation} \label{ujvbkrj}
B^\half f=S^{-1}(G^*+I) f\;\;\forall f\in\dom B^\half,
\end{equation}
where $S=2\RE(B^{-\half})$, $G=B^{*\half}B^{-\half},$ $\dom G=B^{\half}\left(\dom B^\half\cap\dom B^{*\half}\right).$
It is proved in \cite[Theorem 1]{Gomilko} that the operator $G$ is maximal accretive and the equality $\dom B^\half=\dom B^{*\half}$ holds if and only if the operator $G$ is bounded and has bounded inverse.
Moreover, the following theorem has been established:
\begin{theorem}\label{thegom} \cite[Theorem 3]{Gomilko}
Let $B_\gamma=\left(S^{-1}(G^{*\gamma}+I)\right)^2$, $\gamma\in(0,1)$. Then the operator $B_\gamma$ is maximal sectorial with semi-angle $\cfrac{\pi\gamma}{2}$ and the equality $\dom B_\gamma^\half=\dom B_\gamma^{*\half}$ holds if and only if holds the equality $\dom B^\half=\dom B^{*\half}$.
\end{theorem}
The Kato problem was positively solved in \cite{ElstRob} for a class of abstract maximal sectorial operators, in  \cite{K2, AuTch, AHLMcITch, K3, K1, Lions1962} for elliptic differential operators, 
in  \cite{Sku3, Sku2} for some class of  strongly elliptic functional-differential operators.

In this paper yet another approach to the abstract Kato square root problem is proposed. We essentially relay on the properties of contractions and especially of contractions which are linear fractional transformations of $m$-sectorial operators (see Section \ref{jklckfc}).
 The usage of our method allows
\begin{enumerate}
\item [({\bf a})]to find new representations for maximal accretive and maximal sectorial operators and their square roots, see Section \ref{rjhtymrd};
\item [({\bf b})] to give new proofs for known and to establish new equivalent conditions of equality \eqref{jcyjdyt} for maximal accretive and maximal sectorial operators, see Theorem \ref{sqdomrav22};  
\item [({\bf c})] to prove that for an arbitrary maximal accretive operator of the form $\cB=i\cA$, where $\cA$ is a maximal symmetric non-selfadjoint operator, holds the relation $\dom \cB^\half\ne\dom \cB^{*\half}$ (Theorem \ref{maxsym1});
\item [({\bf d})] to give new proofs for main results established  in \cite{Gomilko};
\item [({\bf e})]
to construct a series of new abstract examples of maximal accretive and maximal sectorial operators $B$ such that $\dom B^\half\ne\dom B^{*\half}$ (Section \ref{exampll}), in particular, we construct abstract
 examples of countable and continuum families of unbounded maximal accretive non-sectorial and  maximal sectorial  operators  for which equality \eqref{jcyjdyt} is violated or holds and which possess additional properties.
\end{enumerate}

{\bf Notations.}
Throughout the paper we suppose that the Hilbert spaces are infinite-dimensional, complex and separable.
We use the symbols $\dom T$, $\ran T$, $\ker T$ for
the domain, the range, and the kernel of a linear operator $T$, by $\rho(T)$ and $\sigma(T)$ we  denote the resolvent set and the spectrum of the operator $T$, $\dD$ is the open unit disc,
 the closures of $\dom T$, $\ran T$ are denoted by $\overline{\dom T}$, $\overline{\ran T}$, respectively,
the identity operator in a Hilbert space is denoted by  $I.$ If $\sL \subset \sH$ is a subspace (closed linear manifolds), the orthogonal projection in $\sH$ onto $\sL$ is denoted by~$P_\sL,$ the orthogonal complement $\sH\ominus\sL$ we denote by $\sL^\perp.$ The notation
$T\uphar \cN$ means the restriction of a linear operator $T$ on the
set $\cN\subset\dom T$.
$\dN$ is the set of natural numbers, $\dN_0:=\dN\cup\{0\}$, $\dR_+:=[0,+\infty)$, $\dC$
is the field of complex numbers.
By $\RE(A)$ and $\IM (A)$ we denote the real and imaginary parts, respectively, of an operator $A$, i.e.,
$$\RE (A)=\half(A+A^*), \;\IM(A)=\frac{1}{2i}(A-A^*),\;\dom \RE (A)=\dom \IM(A)=\dom A\cap\dom A^*.$$

\section{Maximal sectorial operators}
The numerical range $W(A)$ of a linear operator $T$ with domain $\dom A$ in a Hilbert space $\sH$ with the inner product $(\cdot,\cdot)$ is given by
\[
  W(A)=\left\{(Au,u):u\in\dom A, \,\|u\|=1\right\}.
\]
By the Toeplitz-Hausdorff theorem (a short proof can be found in \cite{Gust}) the numerical range is a convex set.

A linear operator $A$ in a Hilbert space $\sH$ is called \textit{accretive} \cite{Kato1961}, \cite{Kato}, \cite{SF}
if $\RE(A u,u)\ge 0$ for all $u\in\dom A$, i.e.,
\[
  W(A) \subseteq \{ z\in\dC: \RE z \ge 0\}.
\]
An accretive operator $A$ is called \textit{maximal accretive}, or
\textit{$m$-accretive} for short, if one of the following equivalent conditions is satisfied:
\begin{enumerate}
\def\labelenumi{\rm (\roman{enumi})}
\item $A$ has no proper accretive extensions in $\sH$;
\item $A$ is densely defined and $\ran (A-\lambda I)=\sH$ for some $\lambda\in \dC$ with $\RE\lambda<0$;
\item $A$ is densely defined and closed, and $A^*$ is accretive;
\item $-A$ generates contractive one-parameter semigroup $T(t)=\exp(-tA)$, $t\in\dR_+ $.
\end{enumerate}
If $A$ is $m$-accretive, then $\ker A=\ker A^*$ and hence
$ \ker A\subseteq\dom A\cap\dom A^*.$

An accretive operator $A$ is called \textit{coercive} if there is a positive constant $a$ such that
$$\RE(Af,f)\ge a||f||^2\;\;\forall f\in\dom A.$$
If $A$ is a coercive, then $\ran A$ is a subspace (closed linear manifold) and $A^{-1}\uphar\ran A$ is a bounded accretive operator. If $A$ is $m$-accretive and coercive in $\sH$, then $\ran A=\sH$.
Clearly, if $A$ is accretive operator, then $A-\lambda I$ is coercive accretive for each $\lambda$, $\RE \lambda<0$.

A linear operator $A$ in a Hilbert space $\sH$ is called \emph{sectorial} with vertex $z=0$ and
semi-angle $\alpha \in [0,\pi/2)$, or \emph{$\alpha$-sectorial} for short, if its numerical range is contained in a
closed sector with semi-angle~$\alpha$, i.e.,
\[
  W(A) \subseteq\left\{z\in\dC:|\arg z|\le \alpha\right\}
\]
or, equivalently, $|\IM (A u,u)|\!\le\! \tan\alpha \,\RE(A u,u)$ for all $u\!\in\!\dom A$.
Clearly, an $\alpha$-sectorial operator is accretive; it is called \textit{maximal sectorial}, or \textit{$m$-$\alpha$-sectorial} for short, if
it is $m$-accretive.

The basic definitions and results concerning ses\-qui\-linear forms can
be found in~\cite{Kato}. If ${\mathfrak a}$ is a closed densely defined
sectorial form in the Hilbert space $\sH$, then by the First
Representation Theorem~\cite{Kato}, there exists a unique
$m$-sectorial operator $A$ in $\sH$ associated with ${\mathfrak a}$ in the
following sense: $(Au,v)={\mathfrak a}[u,v],$
for all $u\in\dom A$ and for all $v\in\dom{\mathfrak a}$. The adjoint
operator $A^*$ is associated with the adjoint form
${\mathfrak a}^*[u,v]:=\overline{{\mathfrak a}[v,u]}$. Denote by $A_R$ the nonnegative selfadjoint
operator associated with the real part
${\mathfrak a}_R[u,v]:=\left({\mathfrak a}[u,v]+{\mathfrak a}^*[u,v]\right)/2$ of the form
${\mathfrak a}$. The operator $A_R$ is called the \textit{real part} of $A$. According to the
Second Representation Theorem~\cite{Kato} the equality
$
\dom{\mathfrak a}=\dom A^\half_R$ holds.
Moreover,
\begin{equation}\label{secttt31}
{\mathfrak a}[u,v]=((I+iG)A^{\frac{1}{2}}_Ru,A^{\frac{1}{2}}_Rv),\;u,v\in\dom{\mathfrak a},
\end{equation}
where $G$ is a bounded selfadjoint operator in the subspace
$\overline{\ran{A_R}}$ and $||G||\le \tan\alpha$ iff ${\mathfrak a}$ is
$\alpha$-sectorial.
The operator $A_R$ one can consider as the half of the form--sum \cite{Kato}, i.e., ($\dot+$ denotes the form-sum)
$A_R=\half(A\dot+A^*).$ 

By the first representation theorem \cite{Kato}, 
the associated $m$-sectorial operators $A$ and $A^*$ are given by
\begin{equation}\label{sectt21}
\begin{array}{l}
A =A^\half_R (I+i G)A^\half_R,\; \dom A=\left\{u\in\dom A^\half_R :(I+i G)A^\half_R u\in\dom A^\half_R  \right\}, \\
A^*=A^\half_R (I-i G)A^\half_R,\; \dom A^*=\left\{\phi\in\dom A^\half_R :(I-i G)A^\half_R \phi\in\dom A^\half_R \right\}.
\end{array}
\end{equation}
These relations yield that
\[\begin{array}{l}
\dom A\cap\dom A^*=\left\{u\in\dom A_R :\;GA^\half_R  u\in\dom A^\half_R \right\},\\[2mm]
\RE (A) u:=\cfrac{1}{2}(A+A^*) u=A_R  u\;\forall u\in \dom A\cap \dom A^*,\\[2mm]
\dom A_R\supseteq\dom A\cap\dom A^*$, $A_R\supseteq\RE (A),
\end{array}
\]
 and
\begin{equation}\label{zlhf}
\ker A=\ker A^*=\ker A_R.
\end{equation}
If $A$ is an $m$-accretive (respectively, $m-\alpha$-sectorial) operator and $\ker A=\{0\}$, then the inverse operator $A^{-1}$ is $m$-accretive (respectively, $m-\alpha$ sectorial) as well.

An $m$-sectorial operator $A$ is coercive if and only if the operator $A_R$ is positive definite. In this case
 \[
A^{-1}=A_R^{-\half}(I+iG)^{-1}A_R^{-\half}  
\]
and
\[
(A^{-1})_R=\RE (A^{-1})=A_R^{-\half}(I+G^2)^{-1}A_R^{-\half},\; ((A^{-1})_R)^{-1}=A_R^{\half}(I+G^2)A_R^{\half}.
\]
Hence, for $m-\alpha$-sectorial $A$ we get in the sense of quadratic forms \cite[Chapter VI, § 2, Section 5, Theorem 2.21]{Ka}
\[
A_R\le (\RE (A^{-1}))^{-1}\le \cfrac{1}{\cos^2\alpha} A_R
\]
and then
\begin{equation} \label{hfdhtfk}
\dom A_R^\half=\dom(\RE (A^{-1}))^{-\half}=\ran(\RE (A^{-1}))^\half.
\end{equation}


\section{Contractions of the class $\wt C_\sH$} \label{jklckfc}
As is well known (see e.g., \cite{SF}), if $A$ is $m$-accretive operator, then the linear fractional transformation
\[
Z:=(I-A)(I+A)^{-1}
\]
is a contraction. Conversely, if $Z$ is a contraction, then the linear relation \cite{Ar}
\begin{equation}\label{rhtqy}
A:=-I+2(I+Z)^{-1}=\left\{\left\{(I+Z)h,(I-Z)h\right\}: h\in\sH\right\}
\end{equation}
is $m$-accretive. Moreover, $A$ in \eqref{rhtqy} is  operator if and only if $\ker(I+Z)=\{0\}$ and if this is the case, then
\[
(Af,f)=\left((I+Z^*)(I-Z)h,h\right)=\left((I-Z^*Z)h,h\right)-((Z-Z^*)h,h),\; f=(I+Z)h.
\]
The operator $A$ is bounded if and only if $\ran (I+Z)=\sH$.

Note that for any contraction $Z$ holds the equality $\ker (I-Z)=\ker (I-Z^*)$ (see \cite[Chapter I, Proposition 3.1] {SF}).

The lemma below will be used in the next sections.
\begin{lemma}\label{gjcktte}
(1) Let $Z$ be a bounded operator in the Hilbert space $\sH$. Suppose $\ker (I-Z)=\ker (I+Z)=\{0\}$.
Then
$$\ran (I-Z)\cap\ran (I+Z)=\ran (I-Z^2),$$
and, therefore
the equality $\ran (I-Z^2)=\sH$ is equivalent to the equalities $\ran (I-Z)=\ran (I+Z)=\sH.$

Moreover,
\begin{enumerate}
\def\labelenumi{\rm (\roman{enumi})}
\item [\rm (a)] if $\ran (I-Z)\subseteq\ran (I+Z)$, then $\ran(I+Z)=\sH$,
\item [\rm (b)]if $\ran(I-Z)=\ran (I+Z)$, then $\ran(I-Z)=\ran(I+Z)=\sH.$
\end{enumerate}

(2) If $X$ is a selfadjoint contraction and $\ker (I-X)=\ker(I+X)=\{0\},$ then
\[
\ran (I-X)^\half\cap\ran(I+X)^\half=\ran(I-X^2)^\half.
\]
\end{lemma}

\begin{proof}
(1) Since
\[
I-Z^2=(I-Z)(I+Z)=(I+Z)(I-Z),
\]
the inclusion $h\in\ran (I-Z^2)$ implies inclusions $h\in\ran(I-Z)$ and $h\in\ran (I+Z)$, i.e., $h\in \ran (I-Z)\cap\ran (I+Z)$.

Assume that $h\in \ran (I-Z)\cap\ran (I+Z)$, i.e., $h=(I-Z)f=(I+Z)g$.
Set
\[
\f:=f-g,\;\psi:= f+g
\]
Then
\[
\f=Z\psi,\; f=\half(I+Z)\psi,\; g=\half(I-Z)\psi.
\]
It follows that $h=\half(I-Z^2)\psi,$ i.e., $h\in\ran (I-Z^2)$.

Suppose that $\ran (I-Z)\subseteq\ran (I+Z)$, then
$$\ran (I-Z^2)=\ran (I+Z)\cap\ran (I-Z)=\ran (I-Z).
$$
Hence, from the equality $\ran (I-Z^2)=(I-Z)\ran (I+Z)$
we get $\ran (I+Z)=\sH$.
Consequently, the equality $\ran(I-Z)=\ran (I+Z)$ yields that $\ran(I-Z)=\ran (I+Z)=\sH.$

(2) Let
 $E_X(t),$ $t\in[-1,1]$ be the orthogonal spectral function of $X$. Then due to the equality $(1-t^2)^{-1}=\half\left((1-t)^{-1}+(1+t)^{-1} \right)$ we get the equivalence for a vector $f\in\sH\setminus\{0\}$
\[
\int\limits_{-1}^{1}\cfrac{d(E_X(t)f,f)}{1-t^2}<\infty\Longleftrightarrow \left\{\begin{array}{l}\int\limits_{-1}^{1}\cfrac{d(E_X(t)f,f)}{1-t}<\infty\\
\int\limits_{-1}^{1}\cfrac{d(E_X(t)f,f)}{1+t}<\infty \end{array}\right..
\]
Hence
\[
f\in\dom(I-X^2)^{-\half}\quad\mbox{if and only if}\quad f\in\dom (I-X)^{-\half}\cap\dom (I+X)^{-\half},
\]
i.e., $\ran (I-X)^\half\cap\ran(I+X)^\half=\ran(I-X^2)^\half$.
\end{proof}
\begin{remark}\label{cntgtym}
If $Z$ is a bounded operator, then the factorization
\[
I-Z^{2n+1}=(I-Z)(I+Z+\ldots+Z^{2n}), \; n\in\dN,
\]
yields the implication
\[
\ran (I-Z)\ne \sH\Longrightarrow \ran (I-Z^{2n+1})\ne \sH.
\]

\end{remark}
Further we need contractions which are linear fractional transformations of maximal sectorial linear  operators or linear relations.
\begin{definition} \label{jghtlc} \cite{Arl1987}.
Let $\alpha\in (0,\pi/2)$. A linear operator  $Z$ on the Hilbert space $\sH$  is said to belong to the class
$C_\sH(\alpha)$ if
\[
||Z\sin\alpha\pm i\cos\alpha\, I||\le 1.
\]
\end{definition}
The next proposition immediately follows from Definition \ref{jghtlc}
\begin{proposition} \label{cdjqcndf}(cf. \cite{Arl1987,Arl1991})
The following are equivalent:
\begin{enumerate}
\def\labelenumi{\rm (\roman{enumi})}
\item
$Z\in C_\sH(\alpha)$;
\item $\pm Z^*\in C_\sH(\alpha)$;
\item $Z$ is a contraction and
\begin{equation}
\label{CA2} 2|\IM (Zf,f)|\le\tan\alpha(||f||^2-||Zf||^2),\;
f\in \sH;
\end{equation}
\item holds the inequality
\[
2|\IM (Zf,f)|\le\tan\alpha(||f||^2-||Z^*f||^2),\;
f\in \sH.
\]
\item  the bounded operator $S:=(I+Z)(I-Z^*)=I-ZZ^*+Z-Z^*$ is $\alpha$-sectorial;
\item the linear relation $A$ in \eqref{rhtqy}
is $m-\alpha$-sectorial.
\end{enumerate}
\end{proposition}
Due to
\eqref{CA2} it is natural to denote the set of all selfadjoint
contractions by $C_\sH(0)$. Clearly, $C_\sH(0)=\underset{\alpha\in(0,\pi/2)}{\bigcap}C_\sH(\alpha).$
If $||Z||=\delta<1$, then $Z\in C_\sH(\alpha_\delta)$, where $\alpha_\delta=2\tan^{-1}\delta.$

The numerical range $W(Z)$ of $Z\in C_\sH(\alpha)$ is contained in the intersection $C_\dC(\alpha)$
of two disks on  the complex plane:
\begin{equation}\label{gkjcvy}
C_\dC(\alpha)=\left\{z\in\dC:|z\sin\alpha+i\cos\alpha|\le 1 \wedge|z\sin\alpha-i\cos\alpha|\le 1\right\}.
\end{equation}
Therefore, the operators $I+Z$ and $I-Z$ are sectorial with vertex at the origin and semi-angle $\alpha$.
Besides, since $|\IM z|\le \tan{\frac{\alpha}{2}}$ for all $z\in C_\dC(\alpha)$, the operator $Z=iY$, where $Y$ is a selfadjoint contraction, belongs to $C_\sH(\alpha)$
if and only if $||Y||\le \tan{\frac{\alpha}{2}}$.

Set
\[
\wt C_\sH:=\underset{\alpha\in[0,\pi/2)}{\bigcup}C_\sH(\alpha).
\]
Note that due to \eqref{CA2}, an isometric operator belongs to the class $\wt C_\sH$ if and only if it is additionaly selfadjoint.

 Further for a contraction $Z$ we will use the notation $D_Z:=(I-Z^*Z)^\half.$ As is known \cite{SF} the commutation relations $ZD_Z = D_{Z^*}Z,$ $Z^*D_{Z^*}=D_{Z}Z^*$ hold.
Note that one of the equalities $\ran D_Z=\sH$ or $\ran D_{Z^*}=\sH$ is equivalent to the condition $||Z||<1.$

Properties of operators of the class $\wt C_\sH$ were studied in
\cite{Arl1987, Arl1991}. In particular, the following assertions are valid.
\begin{theorem} \label{cnfhzntj} \cite{Arl1987}.
(1) If $Z\in \wt C_\sH$, then subspaces $\ker D_{Z}$ and $\cran D_{Z}$ reduce $Z,$ the restriction $Z\uphar\ker D_{Z}$ is selfadjoint and unitary operator, the restriction
$Z\uphar\cran D_{Z}$ belongs to the class $C_{00}$ \cite{SF}, i.e.,
 $$\lim\limits_{n\to \infty}Z^nf=\lim\limits_{n\to \infty}Z^{*n}f=0\;\; \forall f\in\cran D_{Z}.$$
(2) If $Z\in C_\sH(\alpha)$, then
$Z^n\in  C_\sH(\alpha)$ and
\begin{equation}\label{defrav}
\ran D_{Z^n}=\ran D_{Z^{*n}}=\ran D_{\RE( Z)}\;\;\forall n\in\dN.
\end{equation}

(3) If $Z_1,Z_2\in C_\sH(\alpha),$ then $\half(Z_1 Z_2+Z_2Z_1)\in C_\sH(\alpha),$ in particular, if $Z_1$ and $Z_2$ commute, then $Z_1,Z_2\in C_\sH(\alpha)\Longrightarrow Z_1Z_2\in C_\sH(\alpha)$.

(4) If $A$ is $m-\alpha$-sectorial operator, then $\exp(-tA)\in C_\sH(\alpha)$ for all $t\in\dR_+$.
\end{theorem}
Note that if $A$ is $m-\alpha$-sectorial operator then, for each $t\in\dR_+$ the numerical range $W(\exp(-tA))$ is contained in the square of the set $C_\dC(\alpha)$ (see  \cite[Theorem 3.4]{ArlZag2010}), i.e.,
\[
W(\exp(-tA))\subseteq\Omega(\alpha):=\{z^2: z\in C_\dC(\alpha)\}.
\]

The next theorem will be used in Section \ref{rjhtymrd}.
\begin{theorem}\label{vyjujtr}
 The following statements are equivalent for $Z\in\wt C_\sH$:
\begin{enumerate}
\def\labelenumi{\rm (\roman{enumi})}
\item $||Z||<1$;
\item $\ran (I+Z)\subseteq\ran D_{Z^*}$ and $\ran (I-Z)\subseteq\ran D_{Z^*}$;
\item $\ran (I+Z)\supseteq\ran D_{Z^*}$ and $\ran (I-Z)\supseteq\ran D_{Z^*}$;
\item $\ran(I+Z)=\ran (I-Z)=\sH$.
\end{enumerate}
\end{theorem}
\begin{proof}
Clearly
$$
||Z||<1 \Longrightarrow \ran (I+Z)=\ran (I-Z)=\ran D_{Z^*}=\sH.$$ 
Suppose (ii) holds. Then by Douglas' lemma \cite{Doug} we get
\[
||(I+Z^*)f||\le c_1||D_{Z^*}f||,\;||(I-Z^*)f||\le c_2||D_{Z^*}f||\;\;\forall f\in\sH.
\]
Since
\begin{equation}
\label{gjktpyj}
||(I\pm Z^*)f||^2+||D_{Z^*}f||^2=2\RE\left((I\pm Z)f,f\right)
=2||(I\pm \RE (Z))^\half f||^2\;\;\forall f\in \sH,
\end{equation}
we get
\[
||(I\pm \RE( Z))^\half f||\le b ||D_{Z^*}f||\;\;\forall f\in \sH.
\]
Hence
$$\ran (I+\RE (Z))^\half\subseteq\ran D_{Z^*},\;\ran (I-\RE( Z))^\half\subseteq\ran D_{Z^*}.$$
On the other side equalities \eqref{gjktpyj} and the Douglas' lemma implies the inclusions
\[
\ran D_{Z^*}\subseteq\ran(I-\RE (Z))^\half,\;\ran D_{Z^*}\subseteq\ran(I+ \RE (Z))^\half.
\]
Thus
\[
\ran D_{Z^*}=\ran(I-\RE (Z))^\half=\ran(I+ \RE (Z))^\half.
\]
Because $Z\in \wt C_\sH$, from \eqref{defrav} and Lemma \ref{gjcktte} we get the equalities
\begin{equation}\label{gjnjve}
\ran D_{Z^*}=\ran D_{\RE( Z)}=\ran (I+\RE (Z))^\half=\ran(I-\RE (Z))^\half.
\end{equation}
Since
\[
\ran D_{\RE( Z)}=(I+\RE (Z))^\half\ran(I-\RE (Z))^\half=(I-\RE (Z))^\half\ran(I+\RE (Z))^\half,
\]
equalities in \eqref{gjnjve} imply
\[
\ran(I+\RE (Z))^\half=\ran(I-\RE (Z))^\half=\ran D_{\RE( Z)}=\sH
\]
Hence $\ran D_{ Z^*}=\sH$ and therefore $||Z||<1$. Thus, (ii) implies (i).

Suppose that (iii) holds. Using \eqref{gjktpyj}, Douglas’ lemma and arguing as above, we get the equalities
\begin{equation}\label{ytjblf}
\ran (I+\RE (Z))^\half=\ran (I+Z),\; \ran (I-\RE (Z))^\half=\ran (I-Z).
\end{equation}
Because the operator $Z$ belongs to the class $\wt C_\sH$, the operators $I\pm Z$ are sectorial and, therefore, admit the representations
\[
\begin{array}{l}
I+ Z=(I+\RE (Z))^\half(I+iF_+)(I+\RE (Z))^\half,\\[2mm]
I-Z=(I-\RE (Z))^\half(I+iF_-)(I-\RE (Z))^\half,
\end{array}
\]
where $F_{\pm}$ are bounded selfadjoint operators. Hence and from  \eqref{ytjblf}
\[
\ran (I\pm \RE(Z))^{\half}=(I\pm \RE(Z))^{\half}(I+iF_{\pm})\ran (I\pm \RE(Z))^{\half}.
\]
Consequently, because $\ran (I+iF_{\pm})=\sH$, we obtain the equalities
$$\ran (I+\RE (Z))^\half=\ran (I-\RE (Z))^\half=\sH.$$
It follows that $\pm 1\in\rho(\RE(Z))$. Therefore $\ran D_{\RE(Z)}=\sH$ and \eqref{defrav} implies that $\ran D_{Z}=\sH$. Therefore $||Z||<1.$ Thus, (iii)$\Longrightarrow$ (i).

Suppose (iv) holds. Then there are $c_\pm>0$ such that $||(I\pm Z^*)f||\ge c_\pm\,||f||$ for all $f\in\sH$.
Consequently,  from \eqref{gjktpyj} we get the equalities
$$\ran (I+\RE (Z))^\half=\ran (I-\RE (Z))^\half=\sH.$$
Now Lemma \ref{gjcktte} and  \eqref{defrav} lead to the equalities $ \ran D_{\RE(Z)}=\ran D_{Z^*}=\sH.$ Hence $||Z||<1$. The proof is complete.
\end{proof}

\begin{proposition}\label{rcgj}
Let $Z$ be a contraction in $\sH$. Then the operators
\begin{equation}\label{aeyrwzz}
Z(t):=Z\exp(-t(I-Z^2)),\; t\in\dR_+
\end{equation}
are contractions and if $Z\in C_\sH(\alpha)$, then $Z(t)\in C_\sH(\alpha)$ for each $t\in\dR_+.$

Moreover, if $\ran(I-Z^2)\ne\sH$, then $\ran (I-Z(t)^2)\ne\sH$ for each $t>0$.
\end{proposition}
\begin{proof}
Because $Z^2$ is a contraction, the operator $I-Z^2$ is accretive. It generates contractive one-parameter $C_0$-semigroup $\{\exp(-t(I-Z^2))\}_{t\in\dR_+}$ \cite{EngNag}.
Hence, for each $t\in\dR_+$ the operator $Z(t)$ is a contraction.

Assume $Z\in C_\sH(\alpha)$. Then $Z^2\in C_\sH(\alpha)$ (see Theorem \ref{cnfhzntj}) and, hence, $I-Z^2$ is $\alpha$-sectorial. From Theorem \ref{cnfhzntj}, statement (3) (see also \cite[Theorem 1, Theorem 2]{Arl1987}) it follows  that
$$\exp(-t(I-Z^2))\in C_\sH(\alpha)\;\; \forall t>0.$$
Because the operators $Z$ and $\exp(-t(I-Z^2))$ commute and belong to the class $C_\sH(\alpha)$, the operator $Z(t)$ belongs to the same class $C_\sH(\alpha)$ (see Theorem \ref{cnfhzntj}).

The operator $I-Z(t)^2$ can be represented as follows
\[
I-Z(t)^2=I-Z^2\exp(-2t(I-Z^2))=(I-Z^2)\left(I-Z^2\sum\limits_{n=1}^\infty \cfrac{(-2t)^n}{n!}(I-Z^2)^{n-1}\right).
\]
Hence $\ran(I-Z^2)\ne \sH$ implies $\ran(I-Z(t)^2)\ne \sH$.
\end{proof}
\begin{remark}\label{norm}
Because the operator $Z$ is bounded, the one-parameter group
$$\exp(-t(I-Z^2)0), \; t\in\dR$$
 is continuous on $\dR$ w.r.t. the operator-norm topology (see \cite[Proposition 3.5]{EngNag}). Hence, the same is true for the operator-valued function $Z(t)$ defined in \eqref{aeyrwzz}.
\end{remark}

\section{New representations of maximal accretive operators and of their square roots }
\label{rjhtymrd}
In the following statements we characterize bounded sectorial operators whose squares are accretive and sectorial operators.
\begin{proposition} \label{pfuun}
Let $\sH$ be a Hilbert space, let $Q$ be a bounded selfadjoint nonnegative operator in $\sH$, $\ker Q=\{0\}$. Suppose that
\begin{equation}\label{quasherm1}
Z \quad\mbox{is a contraction and}\quad QZ=-Z^*Q.
\end{equation}
Then
\begin{enumerate}
\item the operator
\[
T:=Q(I+Z)=(I-Z^*)Q
\]
is $\cfrac{\pi}{4}$--sectorial, $\ker T=\{0\}$ and
$$T^*=Q(I-Z)=(I+Z^*)Q; $$
\item hold the equalities
\begin{equation}\label{ghjgecn}
\ran T=Q\ran (I+Z),\; \ran T^*=Q\ran (I-Z),
\end{equation}
\item the following are equivalent
\begin{enumerate}
\def\labelenumi{\rm (\roman{enumi})}
\item [{\rm (i)}]$\ran T=\ran T^*$,
\item [{\rm (ii)}] $\ran \RE(T)\subseteq\ran T\cap\ran T^*,$
 \item [{\rm (iii)}] $\ran (I+Z)=\ran (I-Z)=\sH$;
\end{enumerate}
\item
the operator $T^2$ takes the form
$$T^2=Q(I+Z)(I-Z^*)Q=(I-Z^*)Q^2(I+Z)=Q\left((I-ZZ^*)+(Z-Z^*)\right)Q$$
and is accretive;
\item the operator $T^2$ is $\alpha$-sectorial if and only if $Z\in C_\sH(\alpha)$;
\item if $\ran Z\cap\ran Q=\{0\}$, $\ker Z=\{0\},$ then
$$\ran T^2\cap \ran T^{*2}=\{0\}.$$
\end{enumerate}
\end{proposition}
\begin{proof}
The equality in \eqref{quasherm1} implies that the operator $i QZ$ is selfadjoint. Therefore for the operator
$$T:=Q(I+Z)=(I-Z^*)Q$$
we have
\[
\RE (T)=Q,\; \IM (T)=-iQZ=iZ^*Q.
\]
Hence, the bounded operator $T$ is accretive. Therefore, $\ker T=\ker T^*\subseteq\ker \RE(T)$. Because $\RE(T)=Q$ and $\ker Q=\{0\}$ we get $\ker T=\ker T^*=\{0\}$. Since $T^*=Q(I-Z)=(I+Z^*)Q$, we get that $\ker (I+Z)=\ker (I-Z)=\{0\}$ and equalities \eqref{ghjgecn} hold. Since
\[
\ran T=Q\ran(I+Z),\; \ran T^*=Q\ran(I-Z),
\]
the equality $\ran T=\ran T^*$ is equivalent to the equality $\ran(I+Z)=\ran (I-Z)$. By Lemma \ref{gjcktte} the latter is equivalent to the equalities $\ran (I+Z)=\ran (I-Z)=\sH$.
Because of the inclusions $\ran Q= \ran \RE(T)\supseteq\ran T,\ran T^*$, the condition $\ran \RE(T)\subseteq\ran T\cap\ran T^*,$ is equivalent to
the equalities $\ran T=\ran T^*=\ran \RE(T)$.

From \eqref{quasherm1} it follows that the operator $Q ZQ^{-1}\uphar\ran Q$ is a contraction in $\sH$, then due to \cite[Theorem 1]{Ka} we get that the operator
$$X:=-iQ^\half ZQ^{-\half} \uphar\ran Q^\half$$
 is contraction.

The equalities
\[
(XQ^\half g,Q^\half h)=-i(QZg,h)=i(Z^*Qg,h)=i (Qg, Zh)=i(Q^\half g,Q^\half Zh)=(Q^\half g,XQ^\half h).
\]
imply that $X$ is essentially selfadjoint.

Due to the equality
$$iQ^\half XQ^\half=Q Z$$
we obtain
\[
\begin{array}{l}
T=Q+QZ=Q-Z^*Q=Q+iQ^\half XQ^\half,\\
T^*=Q-QZ=Q+Z^*Q=Q-iQ^\half XQ^\half,
\end{array}
\]
and
\[
\left|\IM(Tf,f)\right|=\left|(XQ^\half f,Q^\half f)\right|\le ||Q^\half f||^2=\RE (Tf,f),\; f\in\sH.
\]
It follows that the operator $T$ is $\cfrac{\pi}{4}$-sectorial.

For the square $T^2$ one obtains
\[
T^2=Q(I+Z)(I-Z^*)Q,\; \RE(T^2)=Q(I-ZZ^*)Q,\;i\IM(T^2)=Q(Z-Z^*)Q.
\]
Since $Z$ is a contraction, the operator $T^2$ is accretive. Clearly $T^2$ is $\alpha$-sectorial if and only if the operator $(I+Z)(I-Z^*)$ is $\alpha$-sectorial and by   Proposition \ref{cdjqcndf} the latter is equivalent to $Z\in C_\sH(\alpha)$.

Suppose that $\ran Z\cap\ran Q=\{0\}$ and $\ker Z=\{0\}.$

If $T^2f=T^{*2}g$ for some $f,g\in\sH$, then
\[
Q(I-ZZ^*+Z-Z^*)Qf=Q(I-ZZ^*-Z+Z^*)Qg.
\]
Since $\ker Q=\{0\}$, we get the equality
\[
(Z-Z^*)Q(f+g)=(I-ZZ^*)Q(g-f).
\]
Set $\psi:=f+g$, $\f:=g-f$. Then, using the equality $QZ=-Z^*Q$, we have
\[
\begin{array}{l}
(Z-Z^*)Q\psi=(I-ZZ^*)Q\f\Longleftrightarrow ZQ\psi+ZZ^*Q\f=Q\f+Z^*Q\psi\\
 \Longleftrightarrow ZQ\psi-ZQZ\f=Q\f-QZ\psi\Longleftrightarrow ZQ(\psi-Z\f)=Q(\f-Z\psi).
 \end{array}
\]
Taking into account that $\ran Z\cap\ran Q=\{0\}$, we obtain $\f=Z\psi$ and therefore
\[
ZQ(\psi-Z\f)=0\Longleftrightarrow ZQ(I-Z^2)\psi=0
\]
Since $\ker Z=\{0\}$ and $\ker (I-Z^{2})=\{0\}$, we arrive to the equality $\psi=0$. Hence $\f=0$. Therefore $f=g=0$, i.e., $\ran T^2\cap\ran T^{*2}=\{0\}$.
\end{proof}

\begin{remark}\label{normrem}
(1) The following statement follows from Proposition \ref{quasherm1}: \textit{if $Q$ is a bounded nonnegative selfadjoint operator, $\ker Q=\{0\}$, then there is no a nonzero bounded selfadjoint operator $Z$ which anti-commutes with $Q$ ($QZ+ZQ=0$)}.

 Actually, if $Z$ is a such selfadjoint operator, then $\wh Z=||Z||^{-1}Z$ is a selfadjoint contraction  and $Q\wh Z=-\wh ZQ$. The operator $\wh T=Q+Q\wh Z$ is accretive and $\wh T^2=Q(I-\wh Z^2)Q\ge 0,$ i.e. the operator $\wh T^2$ is nonnegative selfadjoint. The uniqueness of maximal accretive square root of a maximal accretive operator (see \cite[Chapter V, Theorem 3.35]{Kato}) yields that $\IM (\wh T)=0$, i.e., $QZ=0$ and hence $Z=0.$

 On the other side the above statement can be derived from the well known Heinz inequality (see e.g. \cite{FFFN}): if $Q_1$ and $Q_2$ are bounded nonnegative selfadjoint operators, then for each bounded operator $Z$ and for an arbitrary $r\in[0,1]$ holds
 \[
 \left\|Q_1Z+ZQ_2\right\|\ge\left\| Q_1^rZQ_2^{1-r}+ Q_1^{1-r}ZQ_2^{r}\right\|.
 \]
In fact, if $Q_1=Q_2=Q$ and if $r=\half,$ then
$
\left\|QZ+ZQ\right\|\ge 2\left\| Q^\half ZQ^\half\right\|.
$
Hence, $\ker Q=\{0\}$ and the equality $QZ+ZQ=0$ yield $Z=0$.

(2) Suppose that the operator $Z$ is skew-selfadjoint, i.e., $Z=iY$, where $Y$ is a sefadjoint contraction. Then \eqref{quasherm1} means that $Y$ commutes with $Q$ and the operator
$T=Q(I+iY)$ is a normal, $\cfrac{\pi}{4}$--sectorial, $T^*=Q(I-iY)$, and $\ran T=\ran T^*=\ran Q$. The operator $T^2=Q(I+iY)^2Q$ is normal and accretive. It is $\alpha$-sectorial if and only if $iY\in C_\sH(\alpha)$. Since $Y$ is a selfadjoint contraction, the latter means that $||Y||\le \tan\frac{\alpha}{2}$.

\end{remark}

Observe, that condition \eqref{quasherm1} implies the equalities
\[
\begin{array}{l}
QZ^{2n}=Z^{*2n}Q\quad\mbox{and}\quad Q(iZ^{2n})=-(iZ^{2n})^*Q,\\[2mm]
QZ^{2n-1}=-(Z^{2n-1})^*Q,\\[2mm]
TZ^{2n}=Z^{*2n}T,\; TZ^{2n-1}=-Z^{*(2n-1)}T\;\;\forall n\in\dN.
\end{array}
\]
Moreover, since the contractive function $Z(t)$, defined in \eqref{aeyrwzz}, admits the representation
\[
Z(t)=Z\exp(-t(I-Z^2))=\exp(-t)\sum\limits_{n=0}^\infty\cfrac{t^n}{n!}Z^{2n+1},\; t\in\dR_+,
\]
one obtains the equalities
$$QZ(t)=-Z(t)^*Q,\; TZ(t)=-Z(t)^*T\;\;\forall t\in\dR_+.$$
Thus, the operators $iQZ^{2n-1}$ and $iQZ(t)$ are selfadjoint for all $n\in\dN$ and all $t\in\dR_+ $.
It follows that the operators
$$T_n=Q(I+Z^{2n-1}),\; T(t)=Q(I+Z(t))$$
are accretive and $\RE (T_n)=\RE (T_t)=Q$ $\forall n\in\dN$, $\forall  t\in\dR_+$.
In addition, the operator-valued function $T(t)$ is continuous on $\dR_+$ w.r.t. the operator-norm topology (see Remark \ref{norm}).

Taking into account Proposition \ref{rcgj} we arrive to the following corollary.
\begin{corollary}\label{cntgtyb}
Let \eqref{quasherm1} be satisfied. Then

(1) assertions of Proposition \ref{pfuun} hold true for the operators
\[
\begin{array}{l}
T_n=Q(I+Z^{2n-1}),\;T^*_n=Q(I-Z^{2n-1}),\;
 T_n^2=Q(I+Z^{2n-1})(I-Z^{*(2n-1)})Q,\\[2mm]
\wt T_n=Q(I+iZ^{2n})=(I+iZ^{*2n})Q,\; \wt T^2_n=Q(I+iZ^{2n})(I+iZ^{*2n})Q,\; n\in\dN \\[2mm]
 T(t)=Q(I+Z(t))=(I-Z(t)^*)Q,\; T(t)^*=Q(I-Z(t))=(I+Z(t)^*)Q,\\[2mm]
\qquad T(t)^2=Q (I+Z(t))(I-Z(t)^*)Q,\; t\in\dR_+.
\end{array}
\]
Besides
$$\RE (T_n)=\RE (\wt T_n)=\RE (T(t))=Q\;\;\forall n\in\dN,\;\; \forall t\in\dR_+$$
and the operator-valued functions $T(t)$ and $T(t)^2$ are continuous w.r.t. the operator-norm topology.

 (2) The operators $\wh T_n:=Q(I+Z^{2n})=(I+Z^{*2n})Q$, $n\in\dN, $ are selfadjoint.
\end{corollary}

\begin{proposition} \label{quasherm2}
Let $T$ be a bounded accretive operator in $\sH$, $\ker T=\{0\}$. Suppose that $T^2$ is also accretive. Then

(1) $T$ is $\cfrac{\pi}{4}$-sectorial; 

(2) the operator
\begin{equation}\label{ukfdysq}
Z:=i(\RE (T))^{-1}\IM (T),
\end{equation}
is a contraction and holds the relation
\[
(\RE (T))Z=-Z^*(\RE (T)).
\]
(3) the operators $T$, $T^*$ and $T^2$ admit the representations
\begin{equation}\label{ttt2}
\begin{array}{l}
 T=(\RE (T))(I+Z)=(I-Z^*)(\RE (T)),\\[2mm]
  T^*=(I+Z^*)(\RE (T))=(\RE (T))(I-Z),\\[2mm]
 T^2=(\RE (T))(I+Z)(I-Z^*)(\RE( T))=(I-Z^*)(\RE(T))^2(I+Z);
\end{array}
\end{equation}
(4) 
the operator
\[
\left\{\begin{array}{l}
\RE (T^{-1})=\cfrac{1}{2}\left(T^{-1}+T^{*-1}\right),\\[3mm]
 \dom \RE (T^{-1})=\dom T^{-1}\cap\dom T^{*-1}=(\RE (T))\ran (I-Z^2)
\end{array}\right.
\]
is selfadjoint; 

(5) the operators $T^{-1}T^*$ and $T^{*-1}T$ are $m$-accretive and hold the equalities
\begin{equation}\label{gomres}
\begin{array}{l}
T^{-1}T^*=(I-Z)(I+Z)^{-1},\;\dom (T^{-1}T^*)=\ran (I+Z)=T^{*-1}(\ran T\cap\ran T^*),\\[2mm]
T^{*-1}T=(I+Z)(I-Z)^{-1},\;\dom (T^{-1}T^*)=\ran (I-Z)=T^{-1}(\ran T\cap\ran T^*);
\end{array}
\end{equation}

(6) 
the following are equivalent:
\begin{enumerate}
\def\labelenumi{\rm (\roman{enumi})}
\item [{\rm (i)}] the operator $T^2$ is $\alpha$-sectorial,
\item [{\rm (ii)}]the operator $T^{-1}T^*$ is $m-\alpha$-sectorial,
\item [{\rm (iii)}] the operator $Z$ belongs to the class $C_\sH(\alpha)$.
\end{enumerate}
\end{proposition}
\begin{proof}
Let
\[
T=\RE (T)+i\IM (T),\; T^*=\RE (T)-i\IM (T),
\]
be Cartesian decompositions of the operators $T$ and $T^*$. Then the Cartesian decomposition of the operator $T^2=(\RE (T)+i\IM (T))^2$ takes the form
\[
T^2=(\RE( T))^2-(\IM (T))^2+i\left((\RE( T))(\IM (T)+ (\IM (T)(\RE( T))\right).
\]
Since $T^2$ is accretive, we get
\[
\RE(T^2f,f)=\|\RE( T)f\|^2-\|\IM (T)f\|^2\ge 0\;\forall f\in\sH,
\]
i.e., $\|\IM (T)f\|\le  \|\RE( T)f\|$ for all $f\in \sH$.
It follows that if $\RE( T)f=0$, then $\IM(T)f=0$ and, consequently, $Tf=0$. Thus, the condition $\ker T={0}$ and accretiveness of $T^2$ imply $\ker \RE(T)=\{0\}$.
Moreover, because $\cran \RE(T)= \sH$, there exists a contraction $V$ in $\sH$ such that
\[
\IM (T)=V(\RE( T))=(\RE( T))V^*.
\]
Hence, the operator $Z$ given by \eqref{ukfdysq} is a contraction, coincides with the operator i$V^*$, and
$(\RE (T))Z=-Z^*(\RE (T))=i\IM (T).$ It follows that
equalities \eqref{ttt2} are valid and $\ker(I\pm Z)=\{0\}$. Besides,  Proposition  \ref{pfuun} yields that $T$ is a $\cfrac{\pi}{4}$-sectorial.

From expressions for $T$ and $T^*$ and Lemma \ref{gjcktte} we get
\[
\ran T\cap \ran T^*=(\RE (T))\left(\ran (I+Z)\cap\ran(I-Z)\right)=(\RE (T))\ran (I-Z^2).
\]
The operator $T^{-1}$ is $m-\cfrac{\pi}{4}$-sectorial operator and from \eqref{ttt2}
\[
T^{-1}=(I+Z)^{-1}(\RE (T))^{-1},\; T^{*-1}=(I-Z)^{-1}(\RE (T))^{-1}.
\]
\[
\dom T^{-1}\cap\dom T^{*-1}=(\RE( T))\ran (I-Z^2).
\]
Let $f=(\RE T)(I-Z^2)g$, $g\in \sH.$ Then
\begin{multline*}
\RE (T^{-1})(\RE (T))(I-Z^2)g=\cfrac{1}{2}\left(T^{-1}+T^{*-1}\right)f\\
=\cfrac{1}{2}\left((I+Z)^{-1}(\RE (T))^{-1}+(I-Z)^{-1}(\RE (T))^{-1}\right)(\RE (T))(I-Z^2)g\\
=\cfrac{1}{2}\left((I-Z)g+(I+Z)g\right)=g,
\end{multline*}
\[
\left(\RE (T^{-1})+I\right)(\RE (T))(I-Z^2)g=\left(I+\RE (T)+Z^*(\RE (T))Z\right)g.
\]
Since the $\RE T$ and $Z^*(\RE T)Z$ are bounded nonnegative selfadjoint operators, we get that $\ran\left(I+\RE (T)+Z^*(\RE (T))Z\right)=\sH$. Hence
\[
\ran \left(\RE (T^{-1})+I\right)=\sH.
\]
Therefore, the operator $\RE (T^{-1})$ is selfadjoint.

From \eqref{ttt2} it follows the equality
\begin{equation}\label{cjvytd}
T^*(I+Z)h=T(I-Z)h=\RE(T)(I-Z^2)h\;\; \forall h\in\sH.
\end{equation}
This leads to \eqref{gomres}.
 Equivalences in (6) follow from Proposition \ref{cdjqcndf} and equalities \eqref{ttt2}, \eqref{gomres}.
\end{proof}
\begin{remark} \label{realpart}

(1) Set $B:=T^{-2}$. Then $T=B^{-\half}$ and statement (4) (a) means that the operator
$$\RE (B^\half )=\cfrac{1}{2}\left( B^\half +B^{*\half} \right),\; \dom \RE (B^\half )= \dom B^\half \cap\dom B^{*\half} $$
is selfadjoint.
This result was established in \cite[Theorem 5.1]{Kato1961} (see Introduction).
 The proof in \cite{Kato1961} is based on the Yosida approximation $A_n:=B^\half \left(I+n^{-1}B^\half \right)^{-1},$ $n=1,2,...$ of the operator $B^\half $.
Using the same approximation, slightly different proof of this fact was given in \cite[Theorem 3.1, Corollary 3.2] {Okazawa}.

(2) Set $G:=T^{*-1}T=B^{*\half} B^{-\half}.$  From \eqref{cjvytd} 
we get that
$$G=(I+Z)(I-Z)^{-1}=-I+2(I-Z)^{-1}.$$
 Hence
$$G^*+I=2(I-Z^*)^{-1},\; G(I-Z)=I+Z,\; G^*(I-Z^*)=I+Z^*.$$
Then equalities in \eqref{ttt2} yield that
\[
\begin{array}{l}
 G^*(\RE (T))Gf=(\RE (T))f\;\;\forall f\in\dom G,\\[2mm]
T=2(G^*+I)^{-1}\RE (T).
\end{array}
\]
Hence, if $S=2\RE (T)=B^{-\half}+B^{*-\half}$, then $B^{\half}= S^{-1}(G^*+I)$ and $G^*SGf=Sf$ for all $f\in \dom G=\ran (I-Z)$. These relations have been established (by an another way) in \cite[Theorem 1]{Gomilko} (see \eqref{ujvbkrj}). In addition, the equality $T=\RE(T)(I+Z)$ yields the equality $B^\half=(G^{-1}+I)S^{-1}.$
\end{remark}

The next theorem summarizes statements of Propositions \ref{pfuun}, \ref{quasherm2}, and Remark \ref{normrem} (2).
\begin{theorem}\label{summar}
There is a one-to-one correspondence between

{\rm({\bf a})} pairs of bounded operators $\left\{Q,Z\right\}$ in $\sH$ satisfying conditions
\begin{enumerate}
\item $Q$ is a nonnegative selfadjoint operator in $\sH$, $\ker Q=\{0\}$,
\item $Z$ is a contraction in $\sH$,
\item $QZ=-Z^*Q$,
\end{enumerate}
and

{\rm({\bf b})} bounded sectorial operators $T$ such that $T^2$ is accretive and $\ker T=\{0\}.$

The correspondence is given by the mappings
\begin{multline*}
\left\{\begin{array}{l}\qquad Q,Z\\
QZ=-Z^*Q\end{array}\right\}\mapsto \left\{\begin{array}{l} \RE (T)=Q,\; \IM (T)=-iQZ,\\
T=Q(I+Z)=(I-Z^*)Q,\; T^*=Q(I-Z)=(I+Z^*)Q,\\
 T^2=Q(I+Z)(I-Z^*)Q=(I-Z^*)Q^2(I+Z)\end{array}\right\},\\
 \left\{\begin{array}{l}T=\RE (T)+i\IM (T)\\
 \RE(T^2)\ge 0\end{array}\right\}\mapsto \left\{Q=\RE (T),\;Z=iQ^{-1}\IM (T)\right\}.
\end{multline*}

The operator $T$ is normal if and only if $Z$ is a skew-selfadjoint ($Z^*=-Z$).

The operator $T^2$ is $\alpha$-sectorial if and only if $Z\in C_\sH(\alpha)$.

\end{theorem}
Further in Theorem \ref{sqdomrav} we give new representations of unbounded  $m$-accretive and $m$-sectorial coercive operators and their square roots. The proof is based on Proposition \ref{pfuun} and Proposition \ref{quasherm2}.
\begin{theorem}\label{sqdomrav}
Let $B$ be an $m$-accretive operator having bounded inverse.
Then
\begin{enumerate}
\item $B$ admits the representations
\begin{equation}\label{novoeghtl}
B=\cL\left((I+Z)(I-Z^*)\right)^{-1}\cL=(I+Z)^{-1}\cL^2(I-Z^*)^{-1},
\end{equation}
where a positive definite selfadjount operator $\cL$ and a contraction $Z$ are given by the relations
\begin{equation} \label{cccont2}
\cL:=(\RE(B^{-\half}))^{-1},
\end{equation}
\begin{equation}\label{cccont}
Z:=i(\RE(B^{-\half}))^{-1}(\IM (B^{-\half})),
\end{equation}
and satisfy the condition
\[
\cL Z^*f=-Z\cL f,\;\;\forall f\in\dom \cL;
\]
\item the square roots $B^\half $, $B^{*\half} $ 
admit the representations
\begin{equation}\label{novoeghtl3}
B^\half =(I+Z)^{-1}\cL=\cL(I-Z^*)^{-1},\; B^{*\half} =(I-Z)^{-1}\cL=\cL(I+Z^*)^{-1};
\end{equation}
\item the operator $B^{\half}B^{*-\half}$ is $m$-accretive and takes the form
\begin{equation}\label{xqe1}
B^{\half}B^{*-\half}=(I-Z)(I+Z)^{-1};
\end{equation}
\item the following are equivalent:
\begin{enumerate}
\item [{\rm (i)}] operator $B$ is $m-\alpha$-sectorial;
\item [{\rm (ii)}]the operator $B^{\half}B^{*-\half}$ is $m-\alpha$-sectorial;
 \item [{\rm (iii)}] $Z\in C_\sH(\alpha)$.
\end{enumerate}
\end{enumerate}
\end{theorem}
In the next theorem we establish criterions for the equality $\dom B^\half=\dom B^{*\half}$.
\begin{theorem}\label{sqdomrav22}
Let $B$ be an unbounded $m$-accretive operator
having bounded inverse and let the operator $Z$ be given by \eqref{cccont}.
Then
the following are equivalent
\begin{enumerate}
\def\labelenumi{\rm (\roman{enumi})}
\item [{\rm (i)}] $\dom B^\half =\dom B^{*\half}, $
\item [{\rm (ii)}]the operator $B^{1/2}B^{*-1/2}$ is bounded and has bounded inverse
\item [{\rm (iii)}] $\ran (I+ Z)=\ran(I-Z)=\sH,$
\item [{\rm (iv)}]$\ran(I-Z^2)=\sH,$
\item [{\rm (v)}] $  \dom (\RE (B^{-\half}))^{-1}\subseteq\dom B^\half\cap \dom B^{*\half}.$
\end{enumerate}

If $B$ is  $m$-sectorial, then the following are equivalent:
\begin{enumerate}
\def\labelenumi{\rm (\roman{enumi})}
\item [\rm (a)]$\dom B^\half =\dom B^{*\half} $;
\item [\rm (b)]$\dom B^\half =\dom (B_R)^\half $;
\item [\rm (c)]$\dom B^{*\half} =\dom (B_R)^\half $;
\item [\rm (d)] $\dom B^\half \subseteq \dom (B_R)^\half $ and $\dom B^{*\half} \subseteq\dom (B_R)^\half $;
\item[\rm (e)] $\dom (B_R)^\half \subseteq\dom B^\half\cap\dom B^{*\half}$;
\item [\rm (f)] $||Z||<1$; 
\item  [\rm (g)]$ \sup\limits_{f\in\sH\setminus\{0\}}\cfrac{\left\|(\RE (B^{-\half}))f\right\|^2}{\RE\left(B^{-1}f,f\right)}<\infty$; 
\item [\rm (h)]$\sup\limits_{f\in\sH\setminus\{0\}}\cfrac{\left\|B^{*-\half}f\right\|^2+\left\|B^{-\half}f\right\|^2}{\RE\left(B^{-1}f,f\right)}<\infty$,
\item [\rm (i)] $\dom (\RE (B^{-\half}))^{-1}=\dom (\RE(B^{-1}))^{-\half}.$
\end{enumerate}
\end{theorem}

\begin{proof} 
Set $T:=B^{-\half}$. Then $T^2=B^{-1}$. We can apply Propositions \ref{pfuun}, \ref{quasherm2}, and Theorem \ref{summar}. Representations
\eqref{novoeghtl}, \eqref{cccont2}, \eqref{novoeghtl3}, \eqref{xqe1} hold true with $Z$ given by \eqref{cccont}.
So we have
\begin{equation}\label{tquft1}
\dom B^\half =\cL^{-1}\ran (I+Z),\;\dom B^{*\half} =\cL^{-1}\ran (I-Z),
\end{equation}
\begin{equation}\label{htfkmyx}
\RE\left( B^{-1}f,f\right)=\RE(T^2f,f)=\left(\cL^{-1}(I-ZZ^*)\cL^{-1}f,f\right)=||D_{Z^*}\cL^{-1}f||^2,\; f\in\sH.
\end{equation}
Hence, from \eqref{hfdhtfk} and from Douglas's lemma \cite{Doug} we get
\begin{equation}\label{domranhalf}
\dom (B_R)^\half = \ran (\RE (B^{-1}))^\half =\cL^{-1}\ran D_{Z^*}.
\end{equation}
Equalities in \eqref{xqe1} and \eqref{tquft1} imply that (i) and (ii) are equivalent.
The equivalences (i) $\Longleftrightarrow$ (iii) and (i) $\Longleftrightarrow$ (iv) follow from Lemma \ref{gjcktte}. Equalities \eqref{cccont2}, \eqref{tquft1}, and Proposition \ref{pfuun}(3) yield the equivalence of (iii) and (v).

Assume now that $B$ is $m$-sectorial.
From Theorem \ref{vyjujtr} and  equalities \eqref{tquft1}, \eqref{domranhalf}, \eqref{xqe1} we get that (a), (b), (c), (d), (e), and (f) are equivalent.

Equality \eqref{domranhalf} implies that $\ran \RE (B^{-1})^\half=\dom \cL^{-1}=\ran \RE (B^{-\half})$ is equivalent to $||Z||<1,$ i.e., (f) and (i) are equivalent.

Condition $||Z||<1$ and \eqref{htfkmyx} imply that for some $0< c\le 1$ holds
\begin{multline*}
\RE\left( B^{-1}f,f\right)\ge c ||\cL^{-1}f||^2=c||(\RE (T))f||^2=c||(\RE (B^{-\half})f||^2=c\left\|\half(B^{-\half}+B^{*-\half})f\right\|^2\\
=\cfrac{c}{4}\left(\left\|B^{*-\half}f\right\|^2+\left\|B^{-\half}f\right\|^2+2\RE(B^{-\half}f,B^{*-\half}f)\right)\\
=\cfrac{c}{4}\left(\left\|B^{*-\half}f\right\|^2+\left\|B^{-\half}f\right\|^2+2\RE(B^{-1}f,f)\right)\;\;\forall f\in\sH.
\end{multline*}
So, (f) $\Longrightarrow$(g), (f) $\Longrightarrow $(h).

Suppose (g) is fulfilled. Then there is $c>0$ such that
\[
 \left\|(\RE (B^{-\half}))f\right\|^2\le c\,{\RE\left(B^{-1}f,f\right)}\;\; \forall f\in\sH.
\]
From \eqref{htfkmyx} and \eqref{cccont2}  
we obtain
\[
||\cL^{-1}h||^2\le c\,||D_{Z^*}\cL^{-1}h||^2=c\,(||\cL^{-1}h||^2-||Z^*\cL^{-1}h||^2),\; \;\forall h\in\sH.
\]
This means that $||Z^*||=||Z||<1$. Consequently, (g)$\Longrightarrow$(f).

Now assume that (h) holds, i.e., there is $c>0$ such that
\[
\left\|B^{*-\half}f\right\|^2+\left\|B^{-\half}f\right\|^2\le c\,{\RE\left(B^{-1}f,f\right)}\;\;\forall f\in\sH.
\]
Then
\[
\begin{array}{l}
4||(\RE(B^{-\half}))f||^2=\left\|B^{*-\half}f\right\|^2+\left\|B^{-\half}f\right\|^2+2\RE( B^{-\half} f, B^{*-\half}f) \\
=\left\|B^{*-\half}f\right\|^2+\left\|B^{-\half}f\right\|^2+2\RE( B^{-1} f,f)\le (c+2){\RE\left(B^{-1}f,f\right)}\;\;\forall f\in\sH.
\end{array}
\]
Hence
\[
||\cL^{-1} f||^2\le \cfrac{c+2}{4}\,||D_{Z^*}\cL^{-1}f||^2\;\;\forall f\in\sH.
\]
It follows that $||Z||<1.$ Therefore (h) $\Longrightarrow $ (f). The proof is complete.
\end{proof}
\begin{remark}\label{gjckt}
The equivalences of {\rm(i)} and {\rm(ii)} in Theorem \ref{sqdomrav} and of {\rm (i)} and {\rm (ii)} in Theorem \ref{sqdomrav22} have been established in \cite[Theorem 1]{Gomilko}. We use another approach.

As has been mentioned in Introduction (see Theorem \ref{lionskato}), equivalences of {\rm (a)}, {\rm(b)}, {\rm(c)} {\rm(d)}, {\rm(e)} were established in \cite{Lions1962} and \cite{Kato1962}. The proofs in \cite{Lions1962} are based on the theory interpolation spaces and in \cite{Kato1962} on the representations of $m$-sectorial operators and their associated quadratic forms \eqref{secttt31}, \eqref{sectt21}.
\end{remark}

Now we formulate and prove an analogue of \cite[Theorem 3]{Gomilko} (see Introduction, Theorem \ref{thegom}).
\begin{theorem}\label{gomthe3}
Let $B$ be an unbounded $m$-accretive operator
having bounded inverse and let the operators $\cL$ and $Z$ be given by \eqref{cccont2} and \eqref{cccont}, respectively.
Set
\begin{equation} \label{vytytyf}
Z_\gamma:=\left((I+Z)^\gamma-(I-Z)^\gamma \right)\left((I+Z)^\gamma+(I-Z)^\gamma \right)^{-1},\; \gamma\in(0,1).
\end{equation}
Then
\begin{enumerate}
\item $Z_\gamma\in C_\sH(\cfrac{\pi\gamma}{2})$ and
\[
\cL Z^*_\gamma f=-Z_\gamma\cL f,\;\;\forall f\in \dom \cL,
\]
\item the operator
\[
B_\gamma=\cL\left((I+Z_\gamma)(I-Z^*_\gamma)\right)^{-1}\cL=(I+Z_\gamma)^{-1}\cL^2(I-Z^*_\gamma)^{-1},
\]
is $m-\cfrac{\pi\gamma}{2}$-sectorial and its square root is given by
\[
B_\gamma^\half=(I+Z_\gamma)^{-1}\cL=\cL(I-Z^*_\gamma)^{-1},
\]
\item  the following are equivalent:
\begin{enumerate}
\def\labelenumi{\rm (\roman{enumi})}
\item [{\rm (i)}] $\dom B^\half =\dom B^{*\half}, $
\item [{\rm(ii)}] $\dom B^\half_\gamma =\dom B^{*\half}_\gamma. $
\end{enumerate}
\end{enumerate}
\end{theorem}

\begin{proof} Let $Q:=\cL^{-1}$. By Theorem \ref{sqdomrav} the operator $R:=B^{-1}$ takes the form
\[
R=Q(I+Z)(I-Z^*)Q
\]
and
\[
R^\half=B^{-\half}=Q(I+Z)=(I-Z^*)Q,\; R^{*\half}=B^{*-\half}=Q(I-Z)=(I+Z^*)Q.
\]
Since $QZ=-Z^*Q$, from  \eqref{lhjcnt}, we obtain the equality
$$Q(I\pm Z)^\gamma =(I\mp Z^*)^\gamma Q,\; \gamma\in(0,1). $$

Set
\[
F:=(I-Z)(I+Z)^{-1}.
\]
Then $F$ is an $m$-accretive operator, moreover, due to \eqref{xqe1} we have the equality $F=B^\half B^{*-\half}.$

If $\gamma\in(0,1)$, then $F^\gamma$
is $m-\cfrac{\pi\gamma}{2}$-sectorial operator and hence
\[
(I-F^\gamma)(I+F^\gamma)^{-1}\in C_\sH(\cfrac{\pi\gamma}{2}).
\]
Clearly
\[
(I-F^\gamma)(I+F^\gamma)^{-1}=\left((I+Z)^\gamma-(I-Z)^\gamma \right)\left((I+Z)^\gamma+(I-Z)^\gamma \right)^{-1}=Z_\gamma.
\]
Since
\[
Z_\gamma\left((I+Z)^\gamma+(I-Z)^\gamma \right)=(I+Z)^\gamma-(I-Z)^\gamma,
\]
we get
\[
QZ_\gamma=-Z^*_\gamma  Q.
\]
Hence, $\cL Z^*_\gamma f=-Z_\gamma\cL f,$ $f\in\dom\cL.$ Besides
\[
\ran(I\pm Z_\gamma)=\ran (I\pm Z)^\gamma.
\]
By Proposition \ref{pfuun}, the operator
\[
R_\gamma:=Q(I+Z_\gamma)(I-Z^*_\gamma)Q
\]
is bounded $\cfrac{\pi\gamma}{2}$-sectorial and
\[
R^\half_\gamma=Q(I+Z_\gamma)=(I-Z^*_\gamma)Q,\;R^{*\half}_\gamma=Q(I-Z_\gamma)=(I+Z^*_\gamma)Q.
\]
Because
\[
\ran (I\pm Z_\gamma)=\sH\Longleftrightarrow\ran (I\pm Z)=\sH,
\]
we get that
\[
\ran R^\half_\gamma =\ran R^{*\half}_\gamma\Longleftrightarrow \ran R^\half=\ran R^{*\half}.
\]

\end{proof}
\begin{remark}\label{gjlheu} The function
\[
w_\gamma:=\cfrac{(1+z)^\gamma-(1-z)^\gamma}{(1+z)^\gamma+(1-z)^\gamma}=\cfrac{1-\left(\cfrac{1-z}{1+z}\right)^\gamma}{1+\left(\cfrac{1-z}{1+z}\right)^\gamma},\; \;z\in\dD,\;\gamma\in(0,1)
\]
is odd, holomorphic in $\dD$ and has values in the set $C_\dC(\cfrac{\pi\gamma}{2})$ \eqref{gkjcvy}. By the Schwarz lemma it has the  representation
\[
w_\gamma(z)=z\psi_\gamma(z),\; z\in\dD,
\]
where $\psi_\gamma(z)$ is even, holomorphic in $\dD$ and $|\psi_\gamma(z)|< 1,$ $z\in\dD$,
\[
\psi_\gamma(z)=\cfrac{w_\gamma(z)}{z},\; z\in\dD\setminus\{0\}, \; \psi_\gamma(0)=\gamma.
\]
 Hence, by means of the functional calculus for contractions \cite{SF} the operator $Z_\gamma$ given by \eqref{vytytyf} can be represented as follows:
\[
Z_\gamma=Z\psi_\gamma(Z).
\]
Therefore, $\ran Z_\gamma\subseteq\ran Z$.
In particular, $Z_\half=Z\left(I+(I-Z^2)^\half\right)^{-1}.$

\end{remark}

\section{Examples}
\label{exampll}

Here we present abstract (counter)examples related to the Kato square root problem.
 Recall that the Hilbert space $\sH$ is supposed to be infinite-dimensional and separable.

\subsection{Square roots of m-accretive operators}
We begin with a general result related to an arbitrary maximal accretive operators of the form $i\cA$, where $\cA$ is a maximal symmetric non-selfajoint operator. This is a generalization of Lions' example \cite{Lions1962} (see \eqref{lionsop}). We apply our methods from Section \ref{rjhtymrd}.

\begin{theorem} \label{maxsym1}
Let $\cA$ be a non-selfadjoint maximal symmetric operator in the Hilbert space $\sH$, $i\in\rho(\cA)$. Then for maximal accretive operator $\cB=i\cA$ holds the inequality $\dom\cB^{\half}\ne \dom\cB^{*\half}$.
\end{theorem}
\begin{proof}
Because $\RE(\cB f,f)=0$ for all $f\in\dom\cB$, the operator $\cB$ is accretive. Since $i\in\rho(\cA)$, we get $-1\in\rho(\cB)$. Therefore $\cB$ is $m$-accretive, the square root
$\cB^\half$ is $m-\cfrac{\pi}{4}$-sectorial. Set $$\cU:=(\cB-I)(I+\cB)^{-1}=(I-i\cA)(I+i\cA)^{-1}.$$
Then $\cU$ is a non-unitary isometry because $\ran \cU$ is a proper subspace of $\sH$. This means that $ 0\in\sigma(\cU)$. The operator $I-\cU$ is accretive and since
\[
I-\cU=2(I+\cB)^{-1},
\]
we obtain the equalities
$$\dom\cB^\half=\dom (I+\cB)^\half= \ran (I-\cU)^{\half},\; \dom\cB^{*\half}=\dom (I+\cB^*)^{\half}= \ran (I-\cU^*)^{\half}.$$
Since $\ker (I-\cU)=\ker(I-\cU^*)=\{0\}$ and the operator $(I-\cU)^\half$ is $m-\cfrac{\pi}{4}$-sectorial, we get (see \eqref{zlhf})
$$\ker\RE((I-\cU)^\half)=\{0\}.$$
Set
\[
\cQ=\RE((I-\cU)^\half), \; \cZ=i \cQ^{-1}\IM ((I-\cU)^\half).
\]
Then $\cQ\cZ=-\cZ^*\cQ$ and
according to Proposition \ref{quasherm2} the operators $(I-\cU)^\half,$ $(I-\cU^*)^{\half}$, $I-\cU$, and $I-\cU^*$ admit the representations
\[
\begin{array}{l}
(I-\cU)^\half=\cQ(I+\cZ)=(I-\cZ^*)Q,\\[2mm]
(I-\cU^*)^{\half}=\cQ(I-\cZ)=(I+\cZ^*)\cQ,\\[2mm]
I-\cU=(\cQ(I+\cZ))^2=((I-\cZ^*)\cQ)^2=\cQ(I+\cZ)(I-\cZ^*)\cQ,\\[2mm]
I-\cU^*=(\cQ(I-\cZ))^2=((I+\cZ^*)\cQ)^2=\cQ(I-\cZ)(I+\cZ^*)\cQ.
\end{array}
\]
Set
$$\cM:=(I+\cZ)(I-\cZ^*),$$
then $I-\cU=\cQ\cM\cQ,$ $I-\cU^*=\cQ\cM^*\cQ$,
\[
\cU=I-\cQ\cM\cQ,\;\cU^*=I-\cQ\cM^*\cQ.
\]
Hence,
\[
\cU^*\cU=I\Longleftrightarrow \cM^*\cQ^2\cM=\cM+\cM^*\Longleftrightarrow \cM^*(\cQ^2\cM-I)=\cM.
\]
Now assume that $ \ran (I-\cU)^{\half}= \ran (I-\cU^*)^{\half}$. This yields (see Proposition \ref{pfuun}) that $\ran (I+\cZ)=\ran(I-\cZ)=\sH$. It follows that
$0\in\rho(\cM)$. Then the equality $\cM^*(\cQ^2\cM-I)=\cM$ implies that $1\in\rho(\cQ^2\cM)$. It follows that $1\in\rho(\cQ\cM\cQ).$ Since $I-\cQ\cM\cQ=\cU$, we get that $0\in\rho(\cU)$. Contradiction, because $\ran \cU\ne\sH.$  Therefore $\ran (I-\cU)^{\half}\ne \ran (I-\cU^*)^{\half}$, i.e., $\dom\cB^\half\ne\dom\cB^{*\half}$.
 \end{proof}
 Note that for the operator $\cB$ in Theorem \ref{maxsym1} the inclusion $\dom \cB \subset\dom \cB^*$ holds. In the next theorem we construct $m$-accretive non-sectorial operator $B_0$ such that
 \begin{equation}\label{xnjvscnh}
 \begin{array}{l}
 \dom B^\half_0\ne\dom B^{*\half}_0,\;  iB^\half_0 B^{*-\half}_0\quad\mbox{is maximal symmetric but non-selfadjoint},\\[2mm]
  \dom B_0\cap\dom B^*_0=\{0\}.
  \end{array}
\end{equation}
Further we will need a result established in \cite[Theorem 5.1]{schmud}: \textit{if $\cR$ is an operator range (the domain of an unbounded selfadjoint operator or a dense linear manifold, which is the range of a bounded nonnegative selfadjoint), then there is a subspace $\sM$ such that}
 $$\sM\cap\cR=\{0\}\quad\mbox{and}\quad \sM^\perp\cap\cR=\{0\}.$$

\begin{theorem}\label{ghbvths}
Let $Q$ be a bounded nonnegative selfadjoint operator, $\ker Q=\{0\}$ and $\ran Q\ne \sH$. Let $\sM$ be a proper subspace in $\sH$ such that $\sM^\perp\cap
\ran Q=\{0\}$.
Set
\begin{equation}\label{jcyjdygh}
T_0=Q+i\left(QP_\sM Q\right)^\half .
\end{equation}
Then
\begin{itemize}
\item the operator $T_0$ is bounded and $\cfrac{\pi}{4}$-sectorial, $\ker  T_0=\{0\}$;
\item $\sM^\perp\cap\ran T_0=\sM^\perp\cap\ran T_0^*=\{0\}$;
\item $\ran T_0\ne \ran T_0^*$;
\item
the square
\[
T_0^2=QP_{\sM^\perp} Q+i\left(Q\left(QP_\sM Q\right)^\half +\left(QP_\sM Q\right)^\half Q\right),
\]
is accretive and non-sectorial.
\end{itemize}
Moreover, the following are equivalent:
\begin{enumerate}
\def\labelenumi{\rm (\roman{enumi})}
 \item $\sM\cap \ran Q=\{0\}$;
\item $\ker\RE(T_0^2)=\{0\}$;
\item $\ran  T_0^2\cap\ran T_0^{*2}=\{0\}$.
\end{enumerate}

Set $B_0:=(T_0^2)^{-1}$. Then
\begin{enumerate}
\item $B_0$ is $m$-accretive but not sectorial,
\item $\dom B_0^\half \ne\dom B_0^{*\half} ,$
\item the operator $B^\half_0 B^{*-\half}_0$ is $m$-accretive,
$$\RE(B^\half_0 B^{*-\half}_0h,h)=0\;\; \forall h\in\dom ( B^\half_0 B^{*-\half}_0),$$
and $\dom B^\half_0 B^{*-\half}_0 \varsubsetneqq\dom (B^\half_0 B^{*-\half}_0)^*$,
\item $\dom B_0\cap\dom B^*_0=\{0\}$ if and only if $\sM\cap\ran Q=\{0\}$.
\end{enumerate}
\end{theorem}
\begin{proof}
Since $\RE (T_0)=Q\ge 0,$ the bounded operator $T_0$ is accretive. The condition $\ker Q=\{0\}$ yields that $\ker T_0=\ker T^*_0=\{0\}$. From \eqref{jcyjdygh} we get
\begin{equation}\label{pfvhtf}
\RE (T_0^2)=Q^2-QP_\sM Q=QP_{\sM^\perp} Q\ge 0.
\end{equation}
Therefore, the operator $T_0^2$ is accretive and $T_0$ is the accretive square root of $T_0^2$. By Proposition \ref{quasherm2} 
the operator $T_0$ is $\cfrac{\pi}{4}$-sectorial.

Equality \eqref{pfvhtf} yields that $\ker \RE (T_0^2)=\{0\}$ if and only if $\ran Q\cap\sM=\{0\}$.
Because
\[
\left\|\left(QP_\sM Q\right)^\half f\right\|^2=\left\|P_\sM Q f\right\|^2\;\;\forall f\in\sH,
\]
the operator $\left(QP_\sM Q\right)^\half $ admits the representation
\[
\left(QP_\sM Q\right)^\half =VP_\sM Q=QV^*,
\]
where $V:\sM\to\sH$ is an isometry. Since $\ran Q\cap\sM^\perp=\{0\}$, we get that $\cran\left(QP_\sM Q\right)^\half =\sH$. Therefore, $\ran V=\sH$ and, acting in $\sH$, the operator
\begin{equation}\label{jghz}
Z_0:=iV^*=i Q^{-1}\left(QP_\sM Q\right)^\half 
\end{equation}
isometrically maps $\sH$ onto $\sM$. Hence, $Z_0$ is a non-unitary isometry in $\sH$, $\ker Z^*_0=\sM^\perp$ and $Z_0$ satisfies the equality
$QZ_0=-Z_0^*Q$. Because $D_{Z_0}=0$ and $D_{Z^*_0}=P_{\sM^\perp}\ne 0,$ the operator $Z_0$ does not belong to the class $\wt C_\sH$ (see Theorem \ref{cnfhzntj} and \eqref{defrav}). Therefore, by Proposition \ref{quasherm2}, the operator $T^2_0$ is not sectorial.

From \eqref{jcyjdygh} it follows that the operators $T_0$ and $T_0^*$ admit representations
$$T_0=Q(I+Z_0),\; T_0^*=Q(I-Z_0).$$
Therefore, $\ker(I+Z_0)=\ker(I-Z_0)=\{0\}$. Applications of Lemma \ref{gjcktte} gives
\[
\ran T_0\cap\ran T_0^*=Q\left(\ran (I+Z_0)\cap\ran (I-Z_0)\right)=Q\ran (I-Z_0^{2}).
\]
If $\ran(I-Z_0)=\sH$, then the operator $\wh S=i(I+Z_0)(I-Z_0)^{-1}$  is a bounded selfadjoint and $Z_0=(\wh S-iI)(\wh S+iI)^{-1}$ is the Cayley transform of $\wh S$. Hence, $\ran Z_0=\sH.$
This contradicts to the inclusion $\ran Z_0=\sM\varsubsetneqq\sH.$ Therefore $\ran (I-Z_0)\ne \sH$ and similarly, $\ran (I+Z_0)\ne \sH$. From Lemma \ref{gjcktte} it follows that
\[
\ran (I+Z_0)\ne \ran (I-Z_0).
\]
Thus, $\ran T_0\ne\ran T_0^*$.

The operators $T_0^2$ and $T_0^{*2}$  take the form
\[
\begin{array}{l}
T_0^2=Q(I+Z_0)(I-Z_0^*)Q=Q\left(P_{\sM^\perp}+(Z_0-Z_0^*)\right)Q,\\[2mm]
T_0^{*2}=Q(I-Z_0)(I+Z_0^*)Q=Q\left(P_{\sM^\perp}-(Z_0-Z_0^*)\right)Q.
\end{array}
\]
Observe that $\ker(Z_0-Z^*_0)=\{0\}$. Actually, if $Z_0f=Z_0^*f,$ then $f=Z_0^{*2}f.$ This implies $f=0$.

If $\sM\cap\ran Q\ne\{0\}$, then there is $f\ne 0$ such that $Qf\in\sM$. Hence $P_{\sM^\perp}Qf=0$ and
\[
T^2_0 f=Q(Z_0-Z^*_0)Qf= T_0^{*2}(-f)\ne 0.
\]

Suppose $\sM\cap\ran Q=\{0\}$. Because $\ran Z_0=\sM$, $\ker Z_0=\{0\}$ and $\ker (I-Z_0^2)=\{0\}$ we can apply Proposition \ref{pfuun} (see statement (5)).
So, $\ran T_0^2\cap \ran T_0^{*2}=\{0\}$.

Because $\ker T_0^2=\{0\}$ and it is accretive and bounded, the inverse operator $B_0:=(T_0^2)^{-1}$ is $m$-accretive. It is not sectorial and we have $\dom B_0^\half \ne\dom B_0^{*\half} .$
Besides $\dom B_0\cap\dom B^*_0=\{0\}$ if and only if $\sM\cap\ran Q=\{0\}$.

Since $Z_0$ is an isometry, relation \eqref{gomres} implies that the operator
\[
F_0:=T_0^{-1}T_0^*=B^\half_0 B^{*-\half}_0
\]
is $m$-accretive, $\RE(F_0 h,h)=0$ for all $h\in\dom F_0=\ran (I+Z_0)$. Besides the operator $iF_0$ is maximal symmetric but non-selfadjoint.
Note that because $iB^\half_0 B^{*-\half}_0$ is maximal symmetric and non-selfadjoint, the inclusion $\dom B^\half_0 B^{*-\half}_0 \varsubsetneqq\dom (B^\half_0 B^{*-\half}_0)^*$ holds. 
\end{proof}
In the corollary below the countable set $\{B_n\}_{n\in\dN}$ of $m$-accretive operators satisfying \eqref{xnjvscnh} is constructed.
\begin{corollary} \label{cdttt}
Let an operator $Q$ and a subspace $\sM$ be given as in Theorem \ref{ghbvths} and let the operator $Z_0$ be defined by \eqref{jghz}.
Then for each $ n\in\dN$
\begin{enumerate}
\item  the operators
\[
T_n:=Q(I+Z_0^{2n+1})=(I- Z_0^{*(2n+1)})Q=Q+(-1)^niQ\left(Q^{-1}\left(QP_\sM Q\right)^\half\right)^{2n+1}
\]
are $\cfrac{\pi}{4}$--sectorial, $\RE (T_n)=Q$, $\ker T_n=\ker (I+Z_0^{2n+1})=\ker(I-Z_0^{*(2n+1)})=\{0\},$
\item $\ran T_n^*\ne \ran T_n^*,$
\item  the operators 
\[
T_n^2=Q(I+Z_0^{2n+1})(I-Z_0^{*(2n+1)})Q
\]
are accretive and non-sectorial,
\item the operators $T_n^{-1}T_n^*$ are $m$-accretive, $\RE(T^{-1}_n T^*_n h,h)=0$ for all $h\in\dom (T_n^{-1}T_n^*)$, and
$\dom(T_n^{-1}T_n^*)\varsubsetneqq\dom (T_n^{-1}T_n^*)^*$
\item if $\sM\cap\ran Q=\{0\}$, then $\ran T^2_n\cap\ran T^{*2}_n=\{0\}$.
\end{enumerate}
Therefore, the operator
$$B_n:= (T_n)^{-2}=\left(Q(I+Z_0^{2n+1})(I-Z_0^{*(2n+1)})Q\right)^{-1}$$
 is $m$-accretive, non-sectorial, and for each $n\in\dN$
\begin{itemize}
\item $\dom B^\half _n\ne\dom B^{*\half}_n$,
\item the operator $B^{\half}_n B^{*-\half}_n$ is $m$-accretive, $\RE (B^{\half}_n B^{*-\half}_n h,h)=0$ for all  $h\in\dom(B^{\half}_n B^{*-\half}_n)$, and
$\dom(B^{\half}_n B^{*-\half}_n)\varsubsetneqq \dom(B^{\half}_n B^{*-\half}_n)^*$,
\item if $\sM\cap\ran Q=\{0\}$, then $\dom B_n\cap\dom B_n^*=\{0\}$.
\end{itemize}
\end{corollary}
\begin{proof}
The operators
$Z_0^{2n+1}$ are isometries and $\ran Z_0^{2n+1}\subset\sM$ for all $n\in\dN $. Since
$QZ_0^{2n+1}=-Z_0^{*(2n+1)}Q$, $n\in\dN$ (see Corollary \ref{cntgtyb}), we can apply arguments in Theorem \ref{ghbvths}.
\end{proof}
The next corollary is proved similarly.
\begin{corollary} \label{cdttt2}
Let an operator $Q$ and a subspace $\sM$ be given as in Theorem \ref{ghbvths} and let the operator $Z_0$ be defined by \eqref{jghz}.
 Set
\[
Z_0(t):=Z_0\exp(-t(I-Z^2_0),\; t\in \dR_+.
\]
Then for each $ t\in\dR_+ $
 the operator
 $$B_0(t):= \left(Q(I+Z_0(t)(I-Z_0(t)^*)Q\right)^{-1}=\left(Q(I+Z_0(t))Q(I+Z_0(t))\right)^{-1}$$ is $m$-accretive.
 The square root takes the form
\[
B_0(t)^\half=\left(Q(I+Z_0(t)\right)^{-1}=\left((I-Z_0(t)^*)Q\right)^{-1},
\]
and
 $$\dom B_0(t)^{\half}\ne \dom B_0(t) ^{*\half}\;\; \forall t\in\dR_+.$$
 In addition, the operator-valued function $B_0(t)$ is norm-resolvent continuous on $\dR_+$.

If $\sM\cap \ran Q=\{0\},$ then $\dom B_0(t)\cap\dom B_0(t)^*=\{0\}$ for each $t\in\dR_+.$

\end{corollary}

Further we construct a continuum family of maximal accretive operators possessing specific properties.

\begin{theorem}\label{normcont}
There exists a norm-resolvent continuous family $\{\cA(\xi)\}_{\xi\in\dC}$ of unbounded and  boundedly invertible $m$-accretive and non-sectorial operators
such that
\begin{itemize}
\item $\dom \cA(\xi)\cap\dom \cA(\xi)^*=\{0\}$ for each $\xi\in\dC$;
\item $\dom \cA(\xi)^\half \ne\dom \cA(\xi)^{*\half}$ for each $\xi\in\dC$;
\item the operator $\RE (\cA(\xi)^{-\half})$ does not depend on $\xi\in\dC$, therefore, the domain of the closed sesquilinear form associated with $m$-sectorial operator $\cA(\xi)^{\half}$ does not depend on $\xi;$
\item the linear manifold $\ran\IM (\cA(\xi)^{-\half})$ is dense in $\sH$ for each $\xi$ and if $\zeta\ne \xi$, then
\[
\begin{array}{l}
\ran\IM (\cA(\xi)^{-\half})\cap\ran \IM (\cA(\zeta)^{-\half})=\{0\},\\[2mm]
\ran\IM (\cA(\xi)^{-\half})\dot+\ran \IM (\cA(\zeta)^{-\half})={\rm Const};
\end{array}
\]
\item  for each $\xi\in\dC$ the operator $\cA(\xi)^{\half} \cA(\xi)^{*-\half}$ is $m$-accretive and, moreover,
\[
\RE\left(\cA(\xi)^{\half} \cA(\xi)^{*-\half}h,h\right)=0\;\; \forall h\in\dom \left(\cA(\xi)^{\half} \cA(\xi)^{*-\half}\right)  
\]
and $ \dom \left(\cA(\xi)^{\half} \cA(\xi)^{*-\half}\right)\varsubsetneqq  \dom \left(\cA(\xi)^{\half} \cA(\xi)^{*-\half}\right)^*.$

\end{itemize}
\end{theorem}
\begin{proof}
Let $Q$ be a bounded nonnegative selfadjoint operator, $\ker Q=\{0\}$ and $\ran Q\ne \sH$. By \cite[Theorem 3.9]{Arl_ZAg_IEOT_2015} there exists a family of subspaces $\{\sM(\xi)\}_{\xi\in\dC}$ such that
\begin{equation} \label{ctvmgj}
\begin{array}{l}
\bullet\; \sM(\xi)\cap\ran Q=(\sM(\xi))^\perp\cap\ran Q=\{0\}\;\;\forall \xi\in\dC,\\
\bullet\; \sM(\xi)\cap\sM(\zeta)=\{0\},\;\sM(\xi)\dot+\sM(\zeta)=\sH,\; \xi\ne\zeta,\\
\bullet\; \sM(\xi)^\perp=\sM(-1/\bar\xi)\; \forall \xi\in\dC\setminus\{0\},\\
\bullet \;\mbox{the ortho-projector-valued function}\; P_{\sM(\xi)}\; \mbox{is continuous on}\; \dC\\
\quad\mbox{w.r.t. the operator-norm topology}.
\end{array}
\end{equation}

Set
\[
\cT(\xi):=Q+i\left(QP_{\sM(\xi)}Q\right)^{\half},\;\xi\in\dC.
\]
Then $\cT(\xi)$ is of the form \eqref{jcyjdygh} and the operator-valued function $\cT(\xi)$, $\xi\in\dC$
 is continuous on $\dC$ w.r.t. the operator-norm topology. The corresponding operator $Z(\xi)=iQ^{-1}\left(QP_{\sM(\xi)}Q\right)^{\half}$ isometrically maps $\sH$ onto $\sM(\xi)$.

 Set $\cA(\xi):=\cT(\xi)^{-2}$. Applying Theorem \ref{ghbvths}, we get that all statements of this theorem hold true.

By definition of $\cT(\xi)$ and $\cA(\xi)$ we have that
\begin{enumerate}

\item holds the equality $\RE (\cA(\xi)^{-\half})=Q$, 
therefore \eqref{hfdhtfk} yields that the domain of the closed sesquilinear form associated with $m$-sectorial operator $\cA(\xi)^{\half}$ does not depend on $\xi;$
\item $\IM (\cA(\xi)^{-\half})=\left(QP_{\sM(\xi)}Q\right)^{\half}$, consequently, $\ran \IM (\cA(\xi)^{-\half})=Q\sM(\xi)$ and from \eqref{ctvmgj}
 \[
\begin{array}{l}
\cran \IM (\cA(\xi)^{-\half})=\sH\;\; \forall \xi,\\
\ran\IM (\cA(\xi)^{-\half})\cap\ran \IM (\cA(\zeta)^{-\half})=\{0\},\\
\ran\IM (\cA(\xi)^{-\half})\dot+\ran \IM (\cA(\zeta)^{-\half})=\ran Q=\ran \RE (\cA(\xi)^{-\half}),\;\xi\ne\zeta.
\end{array}
\]

\end{enumerate}
\end{proof}
\subsection{Square roots of m-sectorial operators}
\begin{theorem} \label{rjynh}
Let $Q$ be a bounded nonnegative selfadjoint operator, $\ker Q=\{0\}$ and $\ran Q\ne \sH$. Suppose $\sM$ is a proper subspace in $\sH$ such that $\sM^\perp\cap\ran Q=\{0\}$ and
let the contractive operator $Z_0$ be given by
$$Z_0=i Q^{-1}\left(QP_\sM Q\right)^\half .$$
For $\gamma\in(0,1)$ set
\[
\begin{array}{l}
Z_\gamma^{(0)}:=\left((I+Z_0)^\gamma-(I-Z_0)^\gamma \right)\left((I+Z_0)^\gamma+(I-Z_0)^\gamma \right)^{-1},\; \gamma\in(0,1),\\[2mm]
Z_\gamma^{(0)}(t):=Z_\gamma^{(0)}\exp(-t(I-(Z_\gamma^{(0)})^2),\; t\in\dR_+,\\[2mm]
B_{\gamma,n}:= \left(Q(I+(Z_\gamma^{(0)})^{2n+1})(I-(Z_\gamma^{(0)})^{*(2n+1)})Q\right)^{-1},\; n\in\dN_0,\\[2mm]
S_\gamma (t):=\left(Q(I+Z_\gamma^{(0)}(t))(I-Z_\gamma^{(0)}(t)^*)Q\right)^{-1},\; t\in\dR_+.
\end{array}
\]
Then the operators $B_{\gamma,n}$ and $S_\gamma(t)$ are unbounded $m-\cfrac{\pi\gamma}{2}$-sectorial for each $n\in\dN_0$ and for each $t\in\dR_+$, the operator-valued function $S_\gamma(t)$ is norm-resolvent continuous on $\dR_+$, and
\begin{equation}\label{ythfd11}
\dom B_{\gamma,n}^\half\ne \dom B_{\gamma,n}^{*\half},\;
\dom S_\gamma(t)^\half \ne\dom S_\gamma(t)^{*\half}.
\end{equation}
If $\sM\cap\ran Q=\{0\}$, then
\begin{equation}\label{hfd22}
\dom B_{\gamma,n}\cap \dom B_{\gamma,n}^*=0,\;\dom S_\gamma(t)\cap\dom S^*_\gamma(t)=\{0\}\;\;\forall n\in\dN_0,\;\forall t\in\dR_+.
\end{equation}
\end{theorem}
\begin{proof}
Due to Theorem \ref{ghbvths}, the operator
\[
B_0:=\left(QP_{\sM^\perp} Q+i\left(Q\left(QP_\sM Q\right)^\half +\left(QP_\sM Q\right)^\half Q\right) \right)^{-1}
=\left(Q(I+Z_0)(I-Z^*_0) Q\right)^{-1}
\]
is $m$-accretive but non-sectorial, and
the inequality $\dom B^\half_0\ne \dom B^{*\half}_0$ holds. This is because
the operator $Z_0$ is a non-unitary isometry, $QZ_0=-Z_0^*Q,$ and
$$\ran (I+Z_0)\ne \ran (I-Z_0)$$
 (see the proof of Theorem \ref{ghbvths}).
 By Theorem \ref{gomthe3}
the operator $Z_\gamma^{(0)}$ belongs to the class $C_\sH(\cfrac{\pi\gamma}{2})$ and  hence
$$(Z_\gamma^{(0)})^{2n+1}, Z_\gamma^{(0)}(t)\in C_\sH(\cfrac{\pi\gamma}{2})$$
 for all natural numbers $n$ and for all $t\in\dR_+$ (see Theorem \ref{cnfhzntj} and Proposition \ref{rcgj}).
Besides
$$Q(Z_\gamma^{(0)})^{2n+1}=-(Z_\gamma^{(0)})^{*(2n+1)}Q,\; QZ_\gamma^{(0)}(t)=-Z_\gamma^{(0)}(t)^*Q.$$
Consequently, by Corollary \ref{cntgtyb} and Theorem \ref{sqdomrav} the operators $B_{\gamma,n}$ and $S_\gamma (t)$ and $m-\cfrac{\pi\gamma}{2}$ sectorial for all $n\in\dN$ and all $t\in\dR_+$.

Since $B_{\gamma,0}= S_\gamma(0)$, due to Theorem \ref{gomthe3} (statement (3)) we have the inequalities
$$\dom B_{\gamma,0}^\half=\dom S_\gamma(0)^\half\ne \dom B_{\gamma,0}^{*\half}= \dom S_\gamma(0)^{*\half}, $$
i.e., $\ran (I+Z_\gamma^{(0)})\ne\ran (I-Z_\gamma^{(0)})$. Applying Lemma \ref{gjcktte}, Remark \ref{cntgtym}, and Proposition \ref{rcgj}, one obtains
for all $n\in\dN$ and for all $t\in\dR_+$ the inequalities
\[
\ran (I+Z_\gamma^{(0)})^{2n+1})\ne \ran (I-Z_\gamma^{(0)})^{2n+1}),\;\ran (I+Z_\gamma^{(t)})\ne\ran (I-Z_\gamma^{(t)}).
\]
Now inequalities \eqref{ythfd11} follow from Proposition \ref{pfuun} and Corollary \ref{cntgtyb}.

Assume $\sM\cap\ran Q=\{0\}$. Taking into account that (see Remark \ref{gjlheu})
$$\ran Z_\gamma^{(0)}\subseteq\sM,\;\ker Z_\gamma^{(0)}=\{0\},\;\ker (I-{Z_\gamma^{(0)}}^2)=\{0\},$$
and applying Proposition \ref{pfuun}, we obtain \eqref{hfd22}.

Since the function $S_\gamma(t)^{-1}$ is continuous w.r.t. operator-norm topology on $\dR_+$, the function $S_\gamma(t)$ is norm-resolvent continuous on $\dR_+$.
\end{proof}

\begin{corollary}\label{normcont2}
Let $\alpha\in (0,\pi/2)$.

(1) There exists a  family $\{\cS_\alpha(\xi)\}_{\xi\in\dC}$ of unbounded and coercive $m-\alpha$-sectorial operators
such that for each $\xi\in\dC$
\begin{itemize}

\item $\dom \cS_\alpha(\xi)\cap\dom \cS_\alpha(\xi)^*=\{0\}$,
\item $\dom \cS_\alpha(\xi)^\half \ne\dom\cS_\alpha(\xi)^{*\half}$,
\item $\RE (\cS_\alpha(\xi)^{-\half})$={\rm Const},
\item $\ran\IM (\cS_\alpha(\xi)^{-\half})\cap\ran \IM (\cS_\alpha(\zeta)^{-\half})=\{0\},\;\xi\ne\zeta.$
\end{itemize}
(2) There exists a norm-resolvent continuous family $\{\cX_\alpha(\xi)\}_{\xi\in\dC}$ of unbounded and coercive $m-\alpha$-sectorial operators
such that
for each $\xi\in\dC$
\begin{itemize}
\item $\dom \cX_\alpha(\xi)\cap\dom \cX_\alpha(\xi)^*=\{0\}$,
\item $\dom \cX_\alpha(\xi)^\half =\dom \cX_\alpha(\xi)^{*\half}={\rm Const}$ (does not depend on $\xi$),
\item the linear manifold $\ran\IM (\cX_\alpha(\xi)^{-\half})$ is dense in $\sH$ for each $\xi$ and if $\zeta\ne \xi$, then
\[
\begin{array}{l}
\ran\IM (\cX_\alpha(\xi)^{-\half})\cap\ran \IM (\cX_\alpha(\zeta)^{-\half})=\{0\},\\
\ran\IM (\cX_\alpha(\xi)^{-\half})\dot+\ran \IM (\cX_\alpha(\zeta)^{-\half})={\rm Const}.
\end{array}
\]
\end{itemize}
\end{corollary}

\begin{proof}
Let $Q$ be a bounded nonnegative selfadjoint operator, $\ker Q=\{0\}$ and $\ran Q\ne \sH$ and let $\{\sM(\xi)\}_{\xi\in\dC}$ be a family of subspaces \eqref{ctvmgj}.
Set
\[
\gamma_\alpha=\cfrac{2\alpha}{\pi},\; a_\alpha=\tan\frac{\alpha}{2},
\]
\[
\begin{array}{l}
Z_0(\xi):=iQ^{-1}\left(QP_{\sM(\xi)}Q\right)^{\half},\\[2mm]
Z_{{\gamma_\alpha}}^{(0)}(\xi):=\left((I+Z_0(\xi))^{\gamma_\alpha}-(I-Z_0(\xi))^{\gamma_\alpha} \right)\left((I+Z_0(\xi))^{\gamma_\alpha}+(I-Z_0(\xi))^{\gamma_\alpha} \right)^{-1},\\[2mm]
Z_{a_\alpha}(\xi):=a_\alpha Z_0(\xi)=ia_\alpha Q^{-1}\left(QP_{\sM(\xi)}Q\right)^{\half},\;\xi\in\dC.
\end{array}
\]

(1)
Define
\begin{multline*}
\cS_\alpha(\xi):=\left(Q(I+Z_{{\gamma_\alpha}}^{(0)}(\xi))(I-Z_{{\gamma_\alpha}}^{(0)}(\xi)^*)Q\right)^{-1}\\
=\left(Q(I+Z_{{\gamma_\alpha}}^{(0)}(\xi)\right)^{-2}=\left((I-Z_{{\gamma_\alpha}}^{(0)}(\xi)^*)Q\right)^{-2},\; \xi\in\dC
\end{multline*}
and apply Theorem \ref{rjynh}.

(2) We have $QZ_{a_\alpha}(\xi)=-Z_{a_\alpha}(\xi)^*Q ,$ $||Z_{a_\alpha}(\xi)||=a_\alpha<1$ and therefore  $Z_{a_\alpha}(\xi)\in C_\sH(\alpha)$, for each $\xi\in\dC$.
Set
\[
 \cX_\alpha(\xi):=\left(Q+ia_\alpha \left(QP_{\sM(\xi)}Q\right)^{\half}\right)^{-2}=\left(Q(I+Z_{a_\alpha}(\xi))(I-Z_{a_\alpha}(\xi)^*)Q
 \right)^{-1}.
\]
By Proposition \ref{pfuun} the operator $\cX_\alpha(\xi)$ is $m-\alpha$-sectorial. Theorem \ref{sqdomrav22} and \eqref{ghjgecn} yield the equalities  $\dom (\cX_\alpha(\xi))^\half =\dom(\cX_\alpha(\xi))^{*\half}=\ran Q$ for all $\xi\in\dC$. Since $\ran Z_{a_\alpha}(\xi)=\sM(\xi)$, from Proposition \ref{pfuun} we get that $\dom \cX_\alpha(\xi)\cap\dom \cX_\alpha(\xi)^*=\{0\}$ for all $\xi\in\dC$.

By definition, we have that hold the following equalities:
\[
\RE (\cS_\alpha(\xi)^{-\half})=\RE (\cX_\alpha(\xi))^{-\half})=Q\;\;\forall \xi.
\]
Therefore, from \eqref{hfdhtfk} we get that domains of the closed sesquilinear forms associated with $\cS_\alpha(\xi)^{\half}$ and $\cX_\alpha(\xi)^\half$ do not depend on $\xi.$

From the  relations
\begin{multline*} \ran Z_{a_\alpha}(\xi)=\sM(\xi),\;
\ran Z_{{\gamma_\alpha}}^{(0)}(\xi)\subseteq\sM(\xi),\\
\ran\IM (\cX_\alpha(\xi)^{-\half})=Q\ran\sM(\xi),\;
 \sM(\xi)^\perp\cap\ran Q=\{0\}\;\;\forall \xi
\end{multline*}
    (see  Remark \ref{gjlheu}) it follows  that  the linear manifold $\ran\IM (\cX_\alpha(\xi)^{-\half})$ is dense in $\sH$ for all $\xi$ and hold the relations
\[
\begin{array}{l}
\ran\IM (\cS_\alpha(\xi)^{-\half})\cap\ran \IM (\cS_\alpha(\zeta)^{-\half})=\{0\},\\
\ran\IM (\cX_\alpha(\xi)^{-\half})\cap\ran \IM (\cX_\alpha(\zeta)^{-\half})=\{0\},\\
\ran\IM (\cX_\alpha(\xi)^{-\half})\dot+\ran \IM (\cX_\alpha(\zeta)^{-\half})=\ran Q,\;\xi\ne\zeta.
\end{array}
\]
\end{proof}

\end{document}